\documentclass[A4paper]{article}

\usepackage{lipsum}        

\usepackage[a4paper,top=3cm,bottom=2cm,left=3cm,right=3cm,marginparwidth=1.75cm]{geometry}
\usepackage[utf8]{inputenc}
\usepackage{authblk}
\usepackage{amsmath}
\usepackage{mathtools}
\usepackage{cases}
\usepackage{bm}
\usepackage{amsfonts}
\usepackage{subfiles}
\usepackage{multicol}
\usepackage{tabularx,enumerate}
\usepackage{subcaption}
\usepackage{tikz}
\usepackage{caption}
\usepackage[absolute,overlay]{textpos}
\usepackage{ifthen}
\usepackage{tcolorbox}
\usepackage{mathrsfs}
\usepackage{scalerel}
\usepackage{graphics}

\usepackage{soul}%to draw a horizontal line on words

\usepackage{comment}

\usepackage{todonotes}

\tcbuselibrary{theorems}
\tcbuselibrary{skins}

\usepackage{algorithm}
\usepackage{algpseudocode}
\usepackage{mycommands}
\usepackage{xcolor}
\usepackage{amsthm,amssymb}
\newtheorem {theorem} {Theorem}[section]

\newtheorem{lemma}[theorem]{Lemma}

\newtheorem{remark}[theorem]{Remark}

\newtheorem{assum}{Assumption}

\theoremstyle{definition}
\newtheorem{definition}{Definition}[section]

\usepackage[toc,page]{appendix} 
\usepackage{hyperref}
\hypersetup{
     colorlinks   = true,
     linkcolor    = blue,
     citecolor    = red
}

\title{Convergence rates of regularized quasi-Newton methods without strong convexity}
\author[*]{Shida Wang}
\author[**]{Jalal Fadili}
\author[*]{Peter Ochs}
\affil[*]{Department of Mathematics and Computer Science, Saarland University, Germany}
\affil[**]{Normandie Universit\'e, ENSICAEN, UNICAEN, CNRS, GREYC, France.}
\begin{document}
\maketitle
\begin{abstract}
    In this paper, we study convergence rates of the cubic regularized proximal quasi-Newton method (\csr) for solving non-smooth additive composite problems that satisfy the so-called Kurdyka-\L ojasiewicz (K\L ) property with respect to some desingularizing function $\phi$ rather than strong convexity. After a number of iterations $k_0$, \csr\  exhibits non-asymptotic explicit super-linear convergence rates. In particular, when $\phi(t)=ct^{1/2}$ with $c$ being a positive constant, \csr\ has a convergence rate of order $\left(\frac{C}{(k-k_0)^{1/2}}\right)^{(k-k_0)/2}$, where $k$ is the number of iterations and $C>0$  is a constant. For functions satisfying the global \L ojasiewicz inequality, the rate becomes global (and non-asymptotic).  To the best of our knowledge, this work presents, for the first time, such convergence rates of regularized (proximal) SR1 methods for minimizing non-convex non-smooth objectives satisfying the K\L\   property. Actually, the rates are novel even in the smooth non-convex case. Notably, we achieve this without employing line search or trust region strategies, without assuming the Dennis--Mor\'e condition, without any assumptions on quasi-Newton metrics and without assuming strong convexity. Moreover, for convex problems, we focus on a more tractable gradient regularized quasi-Newton method (\gsr) which can achieve results similar to those obtained with cubic regularization. We also demonstrate, for the first time, non-asymptotic super-linear convergence rate of \gsr\ for solving convex problems with the help of the \L ojasiewicz inequality instead of strong convexity.
\end{abstract}
\section{Introduction}
Quasi-Newton methods have been studied extensively over decades due to their fast convergence. They have been developed to solve smooth optimization problems by only approximating the Hessian information (second derivative) or its inverse of the objective \cite{dennis1977quasi}, while solely relying on first-order information. Based on this idea, numerous variants have been developed, including BFGS \cite{broyden1970convergence}, SR1 \cite{davidon1991variable,broyden1967quasi}, DFP \cite{davidon1991variable,fletcher1963rapidly}.
 A striking advantage of quasi-Newton-type methods \cite{broyden1970convergence,broyden1973local,dixon1972quasi,fletcher1970new,goldfarb1970family,cartis2011adaptive,byrd1987global} for minimizing some strongly convex smooth function $f$, as compared to generic first order schemes, is their local super-linear convergence.
%To achieve super-linear performance of quasi-Newton methods, Dennis--Mor\'e condition is also widely assumed for solving non-smooth composite problem \cite{chen1999proximal} or even smooth problem \cite{cartis2011adaptive}. 
 However, the convergence rates obtained by all works cited above are asymptotic, i.e., $\lim_{k\to\infty}\frac{\norm[]{x_\kp-x^*}}{\norm[]{x_k-x^*}}\to 0$ or $\lim_{k\to\infty}\frac{\norm[]{\nabla f(x_\kp)}}{\norm[]{\nabla f(x_k)}}\to 0$. Since an asymptotic statement is about the limit as $k\to\infty$, we are also interested in a non-asymptotical analysis where the convergence rate is valid for any $k$ or at least for all $k\geq k_0$ where $k_0$ is some constant. Only a non-asymptotic statement can characterize an explicit upper bound on the error of quasi-Newton methods. Non-asymptotic explicit local super-linear convergence rates for classical quasi-Newton schemes, including BFGS and SR1, applied to strongly convex objectives, were established only recently in \cite{rodomanov2021greedy,rodomanov2022rates,jin2023non,ye2023towards}. The guarantee is, however, only local, meaning that the initial point must be close to the minimum point. To obtain global rates, \cite{rodomanov2024global,jin2024nonAmijo} adopt BFGS with line-search strategies, while \cite{wang2024global} studies the SR1 method on a composite problem with a regularization strategy. 
 
 In this paper, we would like to generalize the results from \cite{wang2024global} to problems without strong convexity. A possible way is to assume that the objective function $f$ has the Kurdyka-\L ojasiewicz (K\L) property \cite{bolte2007lojasiewicz,bolte2008characterizations,li2018calculus,qian2022superlinear}. When $f$ verifies the global Lojasiewicz inequality, convergence rates of first-order methods were obtained in \cite{bolte2017error,attouch2009convergence,frankel2015splitting} both in the convex and non convex case. These are functions verifying a gradient domination inequality \cite{nesterov2006cubic}.

In this paper, we derive non-asymptotic convergence rates by studying cubic- and gradient-regularized quasi-Newton methods from \cite{wang2024global} with the help of the K\L\  property. This paper consists of two parts. In the first part, we provide convergence analysis of cubic regularized PQN on non-convex non-smooth problems under a  general desigularization function $\phi$ or $\phi=ct^{1-\theta}$ for $\theta\in(0,1)$. In particular, when $\theta\in(0,\frac{1}{2})$ in the nonconvex case, we show an explicit super-linear convergence rate $O\left(\frac{C}{k^{1/2}}\right)^{k/2}$ for some $C > 0$ and $k$ large enough. In the second part, we turn to the convex case and use more tractable gradient regularization, rather than cubic regularization, to achieve similar results. The main contribution of this paper is to reveal that K\L\ inequality is the key property for the fast convergence of quasi-Newton methods. Though Dennis--Mor\'e condition is widely assumed \cite{chen1999proximal,stella2017forward,cartis2011adaptive} to show superlinear convergence,  our convergence analysis does not rely on this usually hard-to-verify condition. Furthermore, we have no assumptions on our quasi-Newton metrics either. We also do not adopt any line search or trust region strategies to achieve global convergence.

\section{Related Works}
\paragraph{Regularized Newton method.} 
Although global convergence can be obtained in the convex case while maintaining local fast convergence \cite{Dennis,NoceWrig06} by using line search or trust region strategies, there are regularization strategies to achieve the same goal by adding a positive definite matrix to the Hessian (like in the Levenberg--Marquardt method \cite{LEVENBERG,marquardt1963algorithm}). Thanks to positive definiteness of the metric, Newton-type methods are stabilized at the cost of loosing the local fast convergence rates. The pioneering work by Nesterov and Polyak \cite{nesterov2006cubic} has introduced cubic regularization as another globalization strategy that avoids line search and at the same time shows the standard fast local convergence rates of the pure Newton's method. This cubic regularization works even on non-convex problems such as star-convex functions or gradient dominated functions, namely, those satisfying the condition that for any $x\in\dom f$, $f(x)-\min f\leq C\norm[]{\nabla f(x)}^p$ with some $C>0$, $p>0$. Building on this work, Cartis et al.  \cite{cartis2011adaptive,cartis2011adaptive2} proposed an adaptive cubic regularization approach that retains comparable results while assuming only local Lipschitz continuity of the Hessian. In order to remedy the computational cost of cubic regularization, in \cite{mishchenko2023regularized,doikov2024gradient} a gradient regularization strategy was introduced for solving convex problems. It also allows for global convergence with a super-linear rate of convergence for strongly convex and smooth problems. Motivated by their works, \cite{wang2024global} proposed two regularized SR1 proximal quasi-Newton methods with non-asymptotic super-linear convergence rate on strongly convex problems 
\emph{We will study the non-asymptotic convergence rates of two algorithms from \cite{wang2024global} without the need of strong convexity although it played a crucial role in the proofs of \cite{wang2024global}.}

\paragraph{Quasi-Newton methods.} Rodomanov and Nesterov \cite{rodomanov2021greedy} obtained the first local explicit super-linear convergence rate for greedy quasi-Newton methods. Later, Ye et al. \cite{ye2023towards} obtained the first local explicit super-linear convergence rate for the SR1 quasi-Newton method. Building on this, \cite{wang2024global} studied global rates. Both works significantly inspired our development in this paper. However, all results above require strong convexity of the objective function. \emph{In our paper, we study (cubic, gradient) regularized proximal SR1 methods from \cite{wang2024global}, showing non-asymptotic super-linear rates without strong convexity, or even in the no-convex setting. }

There are a few more works, appearing under similar names to `regularized Newton-type method' \cite{kanzow2022efficient,benson2018cubic,bianconcini2015use,lu2012trust,gould2012updating}. Since they do not consider non-asymptotic convergence rates and are using line search or trust region strategies, we do not describe them in more detail here. As for cubic regularization, in the well-known work \cite{cartis2011adaptive}, the authors also considered asymptotic convergence of cubic regularized quasi-Newton method on non-convex smooth problems. They assumed the Dennis-Moré condition that played a crucial rule in their proofs. There is another way to develop  cubic regularized inexact Newton method without the explicit computation of Hessian via using finite differences  \cite{grapiglia2022cubic}. Quasi-Newton methods with cubic regularization were also studied in \cite{kamzolov2023cubic} for star-convex functions. However,  super-linear convergence was not investigated there.

\paragraph{Proximal quasi-Newton and Newton-type methods.}
One of the earliest work on proximal quasi-Newton (PQN) methods is \cite{chen1999proximal}, where they adopted line search to guarantee global convergence and assume the Dennis--Mor\'e criterion to obtain asymptotic super-linear convergence.  \cite{lee2012proximal} also demonstrated an asymptotic super-linear local convergence rate of the proximal Newton and a proximal quasi-Newton method under the same condition. Both papers were dedicated to the convex case. We believe that the Dennis--Mor\'e criterion is very restrictive and hard to check. Later, \cite{scheinberg2016practical} showed that by a prox-parameter update mechanism, rather than a line search, they can derive a sub-linear global complexity on convex problems. However, to the best of our knowledge, there is no result on non-asymptotic explicit super-linear convergence rates for a proximal quasi-Newton method on composite problems without strong convexity. While we focus on the convergence rate, the efficient evaluation of the proximal mapping with respect to the variable metric is also crucial. There are a few works \cite{scheinberg2016practical,karimi2017imro,becker2012quasi} that are associated to this topic. In particular, \cite{becker2012quasi} proposed a proximal calculus to tackle this difficulty for the SR1 metric, followed by \cite{becker2019quasi,wang2025quasi,kanzow2022efficient}.

\section{Preliminaries}
\subsection{Notation}
We denote $\overline \R=\R\cup{\{+\infty\}}$ and $\R_+=[0,\infty)$. $\R^n$ is equipped with the standard inner product $\scal{u}{v}=u^\top v$ and the associated norm $\vnorm{u}:=\sqrt{u^\top u}$, for any $u$ and $v\in \R^n$. Given symmetric positive definite $M\in\R^{n\times n}$, we denote $\norm[M]{u}=\sqrt{u^\top Mu}$.  Given a matrix $A\in\R^{n\times n}$, $\norm[]{A}$ denotes the matrix norm induced by the Euclidean vector norm. We write $\trace A$ for the trace of the matrix $A\in\R^{n\times n}$. Moreover, $\opid$ denotes the identity matrix in $\R^n$. We adopt the standard definition of the Loewner partial order of symmetric positive semi-definite matrices: $A$ and $B$ be two symmetric positive semi-definite matrices, we say $A\preceq B$ (or $A\prec B$) if and only if for any $u\in\R^n$, we have $u^\top(B-A)u\geq 0$ (or $u^\top(B-A)u> 0$, respectively). 
Let $\map{F}{\R^n}{\overline{\R}}$ be a function. The domain of $F$ is defined as $\mathrm{dom}\ F= \set{x\vert F(x)<+\infty}$.
For any $\alpha, \beta\in(0,\infty]$ with $\alpha<\beta$, we denote the set $\set{x\vert \alpha<F(x)<\beta }$ by $[\alpha<F<\beta]$.  Let $(x_k)_{k\in\N}$ be a sequence in $\R^n$ starting from $x_0$. The set of all clusters of $(x_k)_{k\in\N}$ is denoted by $\omega(x_0)$, i.e.,
$$\omega(x_0) = \{x\in\R^n\vert\  \text{$\exists\, (x_{k_q})_{q\in\N}$ such that $x_{k_q}\to x$ as $q\to\infty$}\}\,.$$
We say a function $\map{F}{\R^n}{\overline \R}$ is lower semicontinuous at $\bar x$ if 
$$\liminf_{x\to\bar x} F(x)\geq F(\bar x)\,,$$
and lower semicontinuous (lsc) on $\R^n$ if this holds for every $\bar x \in\R^n$. 
\subsection{Subdifferentials of non-convex and non-smooth functions}
\begin{definition}[Subdifferentials \cite{rockafellar2009variational}]
    Let $\map{F}{\R^n}{\overline{\R}}$ be a proper and lsc function. 
    \begin{enumerate}
        \item For a given $x\in\dom F$, the Fr\'echet subdifferential of $F$ at $x$, written $\hat \partial F(x)$, is the set of all vectors $u\in\R^n$ which satisfy 
        \begin{equation*}
            \liminf_{y\neq x\ y\to x} \frac{F(y)-F(x)-\scal{u}{y-x}}{\norm[]{y-x}}\geq 0\,.
        \end{equation*}
        For $x\notin \dom F$, we set $\hat\partial F(x)=\emptyset$.
        \item The limiting-subdifferential 
        or simply the subdifferential, of $F$ at $x\in\dom F$, written $\partial F(x)$,  is defined through the following closure process
        \begin{equation*}
            \partial F(x) \coloneqq \set{u \in \R^n : \exists\ x_k\to x,\ F(x_k)\to F(x)\ \text{and}\ u_k \to u\ \text{where $u_k\in \hat\partial F(x_k)$}\ \text{as $k\to\infty$}}\,.
        \end{equation*}
    \end{enumerate}
\end{definition}
We denote $\mathrm{crit} F\coloneqq\set{x\vert \partial F(x)\ni0}$. We now recall the non-smooth Kurdyka--Lojasiewicz (KL) inequality introduced in \cite{bolte2007lojasiewicz}. Let $\eta\in(0,+\infty]$. We denote by $\Phi_\eta$ the class of concave functions $\map{\phi}{[0,\eta)}{\R_+}$ which satisfy the following conditions: $\map{\phi}{ [0,\eta)}{\R_+}$ is continuous; $\phi(0)=0$; $\phi$ is continuously differentiable on $(0,\eta)$ with $\phi^\prime(s)>0$, for any $s\in(0,\eta)$.
 
\begin{definition}[Kurdyka-\L ojasiewicz]  
A proper, lower semi-continuous function $F$ is  said to have the K\L\ property locally at $\bar x \in \mathrm{dom}\partial F$ if there exist $\eta\in(0,+\infty]$, a neighborhood $U$ of $\bar x$ and a continuous concave function $\phi\in\Phi_\eta$ such that the Kurdyka--Lojasiewicz inequality holds, i.e.,
\begin{equation}\label{eq:def_KL}
\phi^\prime \left(F(x)-F(\bar x)\right)\mathrm{dist}\left(0,\partial F(x)\right)\geq 1\,,
\end{equation}
for all $x$ in $U\cap [F(\bar x)<F<F(\bar x)+\eta]$. We say $\phi$ is a desingularizing function for F at $\bar{x}$. 
If $F$ satisfies the K\L\ inequality at any $\bar x\in\mathrm{dom}\ \partial F$, then $F$ is called a K\L\ function. 

\end{definition}
    All tame (definable) functions are K\L\ functions \cite{bolte2007clarke}. In particular, for any $\bar x$, semi-algebraic functions have the K\L\ property at any $\bar x$ with $\phi$ of the form $\phi_{}(t)=ct^{1-\theta}$ for some $c\in\R_+$ and $\theta\in(0,1]$ \cite{bolte2007lojasiewicz}.
    Computing the K\L\ exponent $\theta$ is hard but crucial. We refer the readers to \cite{li2018calculus} for the calculus of K\L\ exponent.
Though the KL inequality is defined locally at each point, it can be uniformized thanks to a result of \cite{bolte2014proximal}. 
\begin{lemma}[Uniformized K\L\ property \cite{bolte2014proximal}]\label{lemma:uni_KL}
    Let $\Omega$ be a compact set and let $\map{F}{\R^n}{\overline{\R}}$ be a proper and lower semi-continuous function. Assume $F$ is constant on $\Omega$ and satisfies the K\L\  property at each point of $\Omega$. Then, there exists $\epsilon>0$, $\eta>0$ and $\phi\in\Phi_\eta$ such that for all $\bar x\in\Omega$ and all $x$ in the set
    \begin{equation}\label{eq:def_uni_kl}
        \set{x\in\R^n\vert\mathrm{dist}(x,\Omega)<\epsilon}\cap [F(\bar x)<F<F(\bar x) +\eta]\,,
    \end{equation}
    one has 
    \begin{equation}\label{ineq:uni-kl}
        \phi^\prime \left(F(x)-F(\bar x)\right)\mathrm{dist}\left(0,\partial F(x)\right)\geq 1\,.
    \end{equation}
\end{lemma}

\section{Problem Setup and Main Results}
In this section, we describe the optimization problem we want to tackle with all standing assumptions and our main results. We postpone the technical parts of the convergence analysis to Section~\ref{Section:conv}. 
\subsection{Problem Setup}
In the whole paper, we consider the optimization problem:
\begin{equation}\label{main:problem}
    \min_{x\in\R^n} F(x)\,,\quad\text{where}\ F(x)\coloneqq g(x) +  f(x)\,,
\end{equation}
where we make the following assumptions:
\begin{assum}\label{ass:main-problem1}
\begin{enumerate}
    \item $F= g+f$ is bounded from below;
    \item $\map{g}{\R^n}{\eR}$ is proper, lsc;
    \item $\map{f}{\R^n}{\R}$ is twice differentiable with $L$-Lipschitz gradient and $L_H$-Lipschitz Hessian.
    \item $\argmin F\neq \emptyset$.
\end{enumerate}
\end{assum}
\begin{assum}\label{ass:main-problem2}
 $F$ is a K\L\ function.
\end{assum}
\begin{remark}
    In our assumptions, neither $f$ nor $g$ is required to be convex.
\end{remark}

\subsection{Algorithms and main results}
Let us first recap the classical symmetric rank-1 quasi-Newton (SR1) update $G_+\coloneqq\text{SR1}(A,G,u)$. Given two $n\times n$ symmetric matrices $A$ and $G$ such that $A\preceq G$ and $u\in\R^n$, we define the SR1 update as follows:
\begin{equation}\label{eq:gen-SR1-update}
    \text{SR1}(A,G,u):=\begin{cases}
        G\,, &\ \text{if}\  (G-A)u=0\,;\\
        G-\frac{(G-A)uu^\top(G-A)}{u^\top(G-A)u}\,, & \ \text{otherwise}\,.
    \end{cases}
\end{equation}
\subsubsection{General objectives}
In order to solve the problem in \eqref{main:problem}, we study a cubic regularized proximal SR1 quasi-Newton method (Algorithm~\ref{Alg:Cubic_PQN}).  
\text{Algorithm~\ref{Alg:Cubic_PQN}} (\csr) is similar to the first algorithm proposed in \cite{wang2024global}. The only but important difference is that we will apply the same algorithm with a restarting strategy for non-convex problems. Step~\ref{alg:SR1_da_PQN-step1} has two cases. If $\trace G_k \leq n\bar\kappa$, Step~\ref{alg:SR1_da_PQN-step1a} is executed. Otherwise, Step~\ref{alg:SR1_da_PQN-step1b} implements a special update step with a fixed metric $L\opid$. 
Step~\ref{alg:SR1_da_PQN-step2}
 defines $J_k$ which is not used explicitly.
Step~\ref{alg:SR1_da_PQN-step4} updates the quasi-Newton metric via SR1 method.

We obtain global convergence with the following local explicit rates.
\begin{algorithm}[htp]
\caption{\csr}\label{Alg:Cubic_PQN}
\begin{algorithmic}
\Require $x_0$, $G_0=L\opid$, $r_{-1}=0$, $\bar \kappa\geq L$\,. 
\State \textbf{Update for $k = 0,\cdots,N$}:
\begin{enumerate}[1.]
\item \label{alg:SR1_da_PQN-step1}
\begin{enumerate}[a.]
    \item \label{alg:SR1_da_PQN-step1a}
    \textbf{If $\trace G_k\leq n\bar\kappa$:} \textbf{update} \begin{equation}\label{alg1:update}
    x_\kp  \in\argmin_{x\in\R^n} g(x) + \scal{\nabla f(x_k)}{x-x_k}+\frac{1}{  2}\norm[ G_k+ L_Hr_\km\opid]{x-x_k}^2 + \frac{L_H}{  3}\norm[]{x-x_k}^3\,.
    \end{equation}
    \textbf{Compute} $u_k=x_\kp-x_k$, $r_k=\norm[]{u_k}$, $\lambda_k=L_H(r_\km+r_{k})$.
    
    \textbf{Compute $\tilde G_\kp= G_k+\lambda_k\opid$} (\textbf{Correction step}).
    
    \item \label{alg:SR1_da_PQN-step1b} 
    \textbf{Otherwise:}
    \textbf{compute}
    \begin{equation}\label{alg1:update2}
        x_\kp \in\argmin_{x\in\R^n} g(x) + \scal{\nabla f(x_k)}{x-x_k}+\frac{L+L_Hr_{k-1}}{  2}\norm[]{x-x_k}^2+\frac{L_H}{  3}\norm[]{x-x_k}^3\,.
    \end{equation}
    
    \textbf{Compute} $u_k=x_\kp-x_k$, $r_k=\norm[]{u_k}$, $\lambda_k=L_H(r_{k-1}+r_{k})$.
    
    \textbf{Compute $\tilde G_\kp= (L+\lambda_k)\opid$ (\textbf{Restart step})} 
\end{enumerate}

    \item \label{alg:SR1_da_PQN-step2}\textbf{Denote but do not use explicitly} $J_k\coloneqq\int_0^1\nabla^2f(x_k+tu_k)dt$\,.
    \item \label{alg:SR1_da_PQN-step3} \textbf{Compute}  
    \[
    \begin{split}
        F^\prime(x_\kp)=&\ \nabla f(x_\kp)-\nabla f(x_k)- {\tilde G_\kp} u_k\,.\\ 
    \end{split}
    \]
    \item \label{alg:SR1_da_PQN-step4} \textbf{Update quasi-Newton metric:}
    \begin{equation*}
        G_\kp= \text{SR1}(J_k, {\tilde G_\kp},u_k)\,.
    \end{equation*}
    \item \label{alg:SR1_da_PQN-step5} \textbf{If $\norm{F^\prime(x_\kp)}=0$}:
    Terminate.
\end{enumerate}
\State \textbf{End}
\end{algorithmic}
\end{algorithm}

\begin {theorem}[]\label{thm:main-cubic}
Let Assumption~\ref{ass:main-problem1} and~\ref{ass:main-problem2} hold. \textcolor{black}{Assume the sequence $(x_k)_{k\in\N}$ generated by \textup{Cubic SR1 PQN} (Algorithm~\ref{Alg:Cubic_PQN}) is bounded. Then, we have} $\norm[]{F^\prime (x_k)}\to 0$, \textcolor{black}{$F(x_k)\to \bar{F}$ for some $\bar{F}\in\R$} and \textcolor{black}{$\mathrm{dist}(x_k,\mathrm{crit}F)\to0$ as $k\to\infty$. }Moreover, there exists $k_0\in\N$ such that for any $k\geq k_0$, \textup{Cubic SR1 PQN} (Algorithm~\ref{Alg:Cubic_PQN}) has a local convergence rate:
 \begin{equation}\label{ineq:main-cubic1}
         \left(\min_{\set{i\in\N\vert k_0\leq i\leq N+k_0}} \norm[]{F^\prime (x_i)}\right)\leq \left(C_1\left(\frac{C_{\mathrm{CR1}}}{N}+\frac{C_{\mathrm{CR2}}}{N^{1/2}}\right)\right)^{N/(N+1)}\norm[]{F^\prime (x_{k_0})}^{1/(N+1)}\,;
    \end{equation}
where $N=k-k_0$, $C_{\mathrm{CR1}}\coloneqq (n+1)L+ n\bar\kappa+2nL_H\textcolor{black}{R}$,  $C_{\mathrm{CR2}}\coloneqq \textcolor{black}{2nL_HC_0}$, $C_0\coloneqq\sqrt{ \frac{3r_{k_0}^{2}}{2}+ M(\phi(F(x_{k_0})-\textcolor{black}{\bar{F}}))}$,  $M\coloneqq\frac{3D}{L_H}$, $D= (2n\bar\kappa+2L+2L_H\textcolor{black}{R})$, $C_1\coloneqq 3\phi(F(x_{k_0})-\textcolor{black}{\bar{F}})$ and \textcolor{black}{$R=(6(F(x_0)-\inf F)/L_H)^{1/3}$}.
\newline
Furthermore, for the case that $\phi$ has the form $\phi(t)=ct^{1-\theta}$ for some $c>0$, \textup{Cubic SR1 PQN} has better convergence rates:
\begin{enumerate}
    \item If $\theta\in(0,\frac{1}{2})$, it has a superlinear convergence rate:
    \begin{equation}\label{ineq1:super-linearrate}
       \norm[]{F^\prime (x_{N+k_0})}\leq\left(C\left(\frac{C_{\mathrm{CR1}}}{N}+\frac{C_{\mathrm{CR2}}}{N^{1/2}}\right)\right)^{N/2}\norm[]{F^\prime(\textcolor{black}{x_{k_0}})}\,;
    \end{equation}
    \item if $\theta=\frac{1}{2}$, it has a  superlinear convergence rate:
    \begin{equation}\label{ineq2:super-linearrate}
    \norm[]{F^\prime (x_{N+k_0})}\leq \left(\frac{3c^2}{4}\left(\frac{C_{\mathrm{CR1}}}{N}+\frac{C_{\mathrm{CR2}}}{N^{1/2}}\right)\right)^{N/2}\norm[]{F^\prime (\textcolor{black}{x_{k_0}})} \,;
\end{equation}
    \item if $\theta\in(\frac{1}{2},1)$,
    \begin{equation}
        \min_{\set{i\in\N\vert k_0\leq i\leq N+k_0}}\norm[]{F^\prime (x_i)}\leq \left( C_2\left(\frac{C_{\mathrm{CR1}}}{N}+\frac{C_{\mathrm{CR2}}}{N^{1/2}}\right)  \norm[]{F^\prime(x_{k_0})}^{1/(\theta N)} \right)^{N/(2+(N-1)(2-\frac{1}{\theta}))}\,,
    \end{equation}
    where $C_2 =3 ((1-\theta)c)^{1/\theta} $ and $C = C_1^{ \frac{1-2\theta}{1-\theta}} C_2^{ \frac{\theta}{1-\theta}}$.
\end{enumerate}

\end {theorem}
\begin {proof}
    See Section~\ref{conv:alg1}.
\end {proof}
\textcolor{black}{
\begin{remark}
    If $F$ is coercive, then $(x_k)_{k\in\N}$ is bounded automatically, since $f(x_k)$ decreases monotonically  and $(x_k)_{k\in\N}$ is guaranteed to stay inside $[f\leq f(x_0)]$ (See Lemma~\ref{ineq:descentVal1}). 
\end{remark}
}
\begin{remark}
    \eqref{ineq1:super-linearrate} and \eqref{ineq2:super-linearrate} do not contain any `` $\min$", which can be understood as the fast growth rate of the function ensuring a sufficient decrease of $\norm[]{\nabla f(x_k)}$ when $\theta \in (0, 1/2]$.
\end{remark}
\textcolor{black}{
\begin{remark}
 It is not necessary for $x_{k+1}$ to be the global minimum point in the update step~\ref{alg:SR1_da_PQN-step1}. To obtain the same theoretical results, it requires a stationary point $x_{k+1}$ satisfying the following inequality:
 \begin{equation}
     g(x_{k+1})+\scal{\nabla f(x_k)}{x_{k+1}-x_k}+\frac{1}{2}\norm[G_k+L_Hr_{k-1}\opid]{x_\kp-x_k} + \frac{L_H}{3}\norm[]{x_\kp-x_k}^3\leq g(x_k)\,.
 \end{equation}
\end{remark}
}
\begin{remark}
    The convergence analysis of Theorem~\ref{thm:main-cubic} relies on the uniformized K\L\ property of  $x$ in Lemma~\ref{lemma:uni_KL}. We can obtain global convergence rates if the K\L\ property is global, i.e. if \eqref{ineq:uni-kl} is satisfied with $\epsilon=+\infty$ and $\eta=+\infty$. For instance, the case where $F$ verifies the global \L ojasiewicz inequality with $\theta=1/2$ has become recently popular in the theory of neural network training.
    In particular when $g=0$, we have $F= f$ and  Algorithm~\ref{Alg:Cubic_PQN} exhibits a non-asymptotic superlinear convergence rate:
        \begin{equation}
            \norm[]{\nabla f (x_{N})}\leq\left(C\left(\frac{\overline C_{\mathrm{CR1}}}{N}+\frac{\overline C_{\mathrm{CR2}}}{N^{1/2}}\right)\right)^{N/2}\norm[]{\nabla f(x_0)}\,.
        \end{equation}
        where $\overline C_{\mathrm{CR1}}\coloneqq (n+1)L+ n\kappa +2nL_H\textcolor{black}{R}$,  $\overline C_{\mathrm{CR2}}\coloneqq {2nL_HC_0}$, $C_0\coloneqq\sqrt{ \frac{3r_{k_0}^{2}}{2}+ M(\sqrt{f(x_{0})-{\min f}})}$,  $M\coloneqq\frac{3D}{L_H}$, $D= (n\bar\kappa+2L+2L_H\textcolor{black}{R})$,  $C_1\coloneqq 3\sqrt{f(x_{0})-{\min f}}$ and \textcolor{black}{$R=(6(f(x_0)-\min f)/L_H)^{1/3}$}.
        
        In fact, a slight modification of the proof of Theorem~\ref{thm:main-cubic} shows that even without using a restarting strategy, the following refined superlinear convergence rate can be obtained for Algorithm~\ref{Alg:SR1_da_PQN_PL} (see Apppendix~\ref{App:PL}):
        \begin{equation}
            \norm[]{\nabla f (x_{N})}\leq\frac{\mu}{6}\left(\frac{\Theta}{N^{1/3}}\right)^{N/2}\norm[]{\nabla f(x_0)}\,,
        \end{equation}
        where $\Theta= \left( 2n L+2nL_H\textcolor{black}{R}+ 2nL_H \left(\frac{6}{L_H}\left(f(x_0)-\inf f\right)\right)^{1/3} \right)$ and \textcolor{black}{$R=(6(f(x_0)-\min f)/L_H)^{1/3}$}. 
\end{remark}

\begin{remark}
    When $\phi(t)=ct^{1/2}$, \eqref{ineq2:super-linearrate} indicates that the super-linear rate is attained after $N+k_0\geq \max\{C_{\mathrm{CR1}},C_{\mathrm{CR2}}^2\}+k_0$ number of iterations. Here $ C_{\mathrm{{CR1}}}=O(n)$ and $ C_{\mathrm{{CR2}}}=O(n^{3/2})$ where $n$ is the number of dimension. Therefore, the super-linear rate is attained after $\Omega(\textcolor{black}{n^3})+k_0$ iterations.
\end{remark}

\begin{remark}
We compare our results with convergence rates of first-order methods. When $N$ is sufficiently large, for a general $\phi$, the \textup{Cubic SR1 PQN} has a local sublinear convergence rate $O\left(\frac{1}{N^{1/2}}\right)$. Now, we consider the case when $\phi=ct^{1-\theta}$. For $\theta\in(0,\frac{1}{2}]$, first-order methods achieve linear convergence rates \cite{li2018calculus} while our Cubic SR1 method attains a super-linear convergence rate. When $\theta\in\left(\frac{1}{2},1\right)$, the first-order methods have sublinear convergence rates $O\left(\frac{1}{N^{\frac{1-\theta}{2\theta-1}}}\right)$ (see \cite{attouch2009convergence}), whereas our Cubic SR1 method has a convergence rate close to $O\left(\frac{1}{N^{\frac{\theta}{2(2\theta-1)}}}\right)$. For $\frac{1}{2}<\theta< \frac{2}{3}$, the convergence rate of our Cubic SR1 PQN is slower than that of first-order methods \cite{frankel2015splitting}, but for $\frac{2}{3}<\theta\leq 1$ our Cubic SR1 PQN achieves a faster convergence rate. \textcolor{black}{This phenomenon appears because if a function is too flat, then (close to) degenaracy of the Hessian seems to penalize SR1 and thus does not bring any additional benefit.}

\end{remark}
\subsubsection{Convex objectives}
If we additionally assume that $f$ and $g$ are convex functions, then $F$ is convex and we can apply a gradient regularized quasi-Newton method. That is \text{Algorithm~\ref{Alg:Grad_PQN}} which is the same as the third algorithm proposed in \cite{wang2024global}. We recap the algorithm here for the readers' convenience.
\begin{algorithm}[htp]
\caption{Grad SR1 PQN}\label{Alg:Grad_PQN}
\begin{algorithmic}
\Require $x_0$, $G_0=L\opid$, $\lambda_0=0$, $\tilde G_0=G_0+\lambda_0\opid$, $\bar \kappa\geq L$\,. 
\State \textbf{Update for $k = 0,\cdots,N$}:
\begin{enumerate}[1.]
\item \label{alg:SR1_da_grad_PQN2-step1}\textbf{Update} \begin{equation}\label{alg_grad:update}
    x_\kp =\argmin_{x\in\R^n} g(x) + \scal{\nabla f(x_k)}{x-x_k}+\frac{1}{2}\norm[\tilde G_k]{x-x_k}^2\,.
\end{equation}
    \item \label{alg:SR1_da_grad_PQN2-step2}\textbf{Denote but do not use explicitly} $J_k\coloneqq\int_0^1\nabla^2f(x_k+tu_k)dt$\,.
    \item \label{alg:SR1_da_grad_PQN2-step3}\textbf{Compute} $u_k=x_\kp-x_k$, $r_k=\norm[]{u_k}$, and  
    \[
    \begin{split}
        F^\prime(x_\kp)=&\ \nabla f(x_\kp)-\nabla f(x_k)-\tilde G_{k}u_k\,,\\ 
        G_\kp =&\ \text{SR1}(J_k,\tilde G_k,u_k)\,,\\
        \lambda_\kp=&\ \left(\sqrt{L_H\norm[]{F^\prime (x_\kp)}}+L_Hr_{k}\right)\,.
    \end{split}
    \]
    \item \label{alg:SR1_da_grad_PQN2-step4}\textbf{Update $\tilde G_\kp$:} compute $\hat G_\kp= G_\kp+\lambda_\kp\opid$ (\textbf{Correction step}).
    \begin{enumerate}[]
        \item \textbf{If $\trace{\hat G_\kp}\leq n\bar\kappa$:} we set $\tilde G_\kp= \hat G_\kp$\,.
        \item \textbf{Othewise:} we set $\tilde G_\kp= L\opid$ (\textbf{Restart step}).
    \end{enumerate}
    \item \label{alg:SR1_da_grad_PQN2-step5} \textbf{If $\norm{F^\prime(x_\kp)}=0$}:
    Terminate.
\end{enumerate}
\State \textbf{End}
\end{algorithmic}
\end{algorithm}

For  gradient regularized proximal quasi-Newton methods, we obtain global convergence with the following non-asymptotic rates.
\begin {theorem}[Convex objectives]\label{thm:main-grad}
Let Assumption~\ref{ass:main-problem1} and~\ref{ass:main-problem2} hold. \textcolor{black}{Assume that the sequence $(x_k)_{k\in\N}$ generated by \textup{Grad SR1 PQN} is bounded, i.e. there exists $R$ such that $R\geq 2\norm[]{x_k}$ for any $k\in\N$.} Additionally, we assume both $g$, $f$ are  convex. For any initialization $x_0\in\R^n$ \textup{Grad SR1 PQN} (Algorithm~\ref{Alg:Grad_PQN}) has $\norm[]{F^\prime (x_k)}\to 0$, \textcolor{black}{ $F(x_k)\to \min F$} and \textcolor{black}{$\mathrm{dist}(x_k,\arg\min F)\to0$ as $k\to\infty$. }\textcolor{black}{Moreover, there exists $k_0\in\N$ such that for any $k\geq k_0$, \textup{Grad SR1 PQN} (Algorithm~\ref{Alg:Grad_PQN}) has a local non-asymptotic rate:}
 \begin{equation}\label{convx_conv_Rate:KL}
         \left(\min_{\set{i\vert k_0\leq i\leq N+k_0}} \norm[]{F^\prime (x_i)}\right)\leq \left(S_1\left(\frac{n\bar\kappa}{N}+\frac{C_{\mathrm{GR}}}{N^{1/3}}\right)\right)^{N/(N+1)}\norm[]{F^\prime ({x_{k_0}})}^{1/(N+1)}\,;
    \end{equation}
where $N=k-k_0$, $C_{\mathrm{GR}}\coloneqq nDC_0^{1/3}$, $C_0\coloneqq \frac{r_{k_0}^{3/2}}{3}+ \sqrt{\frac{4(n\bar\kappa+L)}{L_H}}(\phi(F(\textcolor{black}{x_{k_0}})-\textcolor{black}{\min F}))$, $D \coloneqq (\sqrt{L_H  (n\bar\kappa +  L)} + L_H\textcolor{black}{R})$, and $S_1=2\phi(F(\textcolor{black}{x_{k_0}})-\textcolor{black}{\min F})$.
\newline
Furthermore, for the case $\phi$ has the form $\phi(t)=ct^{1-\theta}$ for some $c>0$, \textup{Grad SR1 PQN} has convergence rates:
\begin{enumerate}
    \item if $\theta\in(0,\frac{1}{2})$, it has a super-linear convergence rate:
     \begin{equation}
        \norm[]{F^\prime (x_{N+k_0})}\leq\left(S\left(\frac{n\bar\kappa}{N}+\frac{C_{\mathrm{GR}}}{N^{1/3}}\right)\right)^{N/2}\norm[]{F^\prime (\textcolor{black}{x_{k_0}})}\,,
    \end{equation}
    \item if $\theta=\frac{1}{2}$, it has a super-linear  convergence rate:
    \begin{equation}\label{ineq_grad2:super-linearrate}
    \norm[]{F^\prime (x_{N+k_0})}\leq \left(\frac{c^2}{2}\left(\frac{n\bar\kappa}{N}+\frac{C_{\mathrm{GR}}}{N^{1/3}}\right)\right)^{N/2}\norm[]{F^\prime (\textcolor{black}{x_{k_0}})} \,,
\end{equation}
    \item if $\theta\in(\frac{1}{2},1)$, it has a sub-linear convergence rate:
    \begin{equation}
        \min_{\set{i\in\N\vert k_0\leq i\leq N+k_0}}\norm[]{F^\prime (x_i)}\leq \left(  S_2\left(\frac{n\bar\kappa}{N}+\frac{C_{\mathrm{GR}}}{N^{1/3}}\right)  \norm[]{F^\prime (\textcolor{black}{x_{k_0}})}^{1/(\theta N)} \right)^{1/(\frac{2}{N}+(N-1)(2-\frac{1}{\theta}))}\,,
    \end{equation}
     where  $S_2 =2((1-\theta)c)^{1/\theta}$ and $S = S_1^{ \frac{1-2\theta}{1-\theta}} S_2^{ \frac{\theta}{1-\theta}}$.
\end{enumerate}

\end {theorem}
\begin {proof}
    See Section~\ref{conv:alg2}.
\end {proof}

\begin{remark}
 $(x_k)_{k\in\N}$ is bounded under coercivity of $F$.
\end{remark}

\begin{remark}
    The convergence analysis of Theorem~\ref{thm:main-grad} also relies on the uniformized K\L\ property of $F$ with respect to $\arg\min F$. \eqref{convx_conv_Rate:KL} can be made global if $F$ verifies a global version of the K\L\ inequality, i.e. with $\epsilon=+\infty$ and $\eta=+\infty$.
\end{remark}
\textcolor{black}{
\begin{remark}
    In fact, in the case where $\phi(t)=ct^{1/2}$, \eqref{ineq2:super-linearrate} indicates that the super-linear rate is attained after $N+k_0\geq \max\{n\bar\kappa ,C_{\mathrm{GR}}^3\}+k_0$ number of iterations and $ C_{\mathrm{GR}}=O(n^{5/3})$. Therefore, the super-linear rate is attained after $\Omega(n^5)+k_0$ iterations.
\end{remark}
}

Let $g=0$. Let $f$ satisfy Assumption~\ref{ass:main-problem1} and the gradient domination inequality (global \L ojasiewicz inequality): 
\begin{equation}\label{PL}
    f(x)-\textcolor{black}{\min f}\leq \frac{c}{2}\norm[]{\nabla f(x)}^2\,,
\end{equation} 
for some $c>0$ and any $x$. A better global convergence rate can be retrieved.
\begin {theorem}[]\label{thm2:main-grad}
Let Assumption~\ref{ass:main-problem1} and~\ref{ass:main-problem2} hold. Additionally, we assume $g=0$ and $f$ is a convex function satisfying \eqref{PL}. For any initialization $x_0\in\R^n$ \textup{Grad SR1 PQN} (Algorithm~\ref{Alg:Grad_PQN}) has $\norm[]{\nabla f (x_N)}\to 0$ as $N\to+\infty$ and \textup{Grad SR1 PQN} (Algorithm~\ref{Alg:Grad_PQN}) has a global non-asymptotic rate:
 \begin{equation}\label{ineq_PL:super-linearrate}
         \norm[]{\nabla f (x_N)}\leq \left(\frac{c^2C_\mathrm{GD}}{2N^{}}\right)^{N/(2)}\norm[]{\nabla f (\textcolor{black}{x_{0}})}^{2/(N+1)}\,;
    \end{equation}
where  $C_{\mathrm{GD}}\coloneqq ( n\bar\kappa+n\sqrt{L_H}C_f+nL_HC_r  )$, $C_f=4n\bar\kappa(cn\bar\kappa(f(x_0)-\textcolor{black}{\inf f}))^{1/4}$ and $$C_r=\frac{1}{\sqrt{L_H}}4cn\bar\kappa(2n\bar\kappa(f(x_0)-\textcolor{black}{\inf f}))^{1/4}+ \frac{1}{L}2cn\bar\kappa(2n\bar\kappa(f(x_0)-\textcolor{black}{\inf f}))^{1/2}\,.$$
\end {theorem}
\begin {proof}
    See Section \ref{conv:alg2}.
\end {proof}
\begin{remark}
 \eqref{ineq_PL:super-linearrate} indicates that the super-linear rate is attained after $N\geq C_{\mathrm{GD}}$ number of iterations and $ C_{\mathrm{GD}}=O(n^{2.5})$.
\end{remark}

\section{Convergence Analysis}\label{Section:conv}
In this section, without loss of generality, we assume that $\norm[]{F^\prime(x_k)}>0$ for any $k\in\N$, since, otherwise, our algorithms terminate after a finite number of steps and our non-asymptotic rates would hold until then. For convenience, we denote $F^\prime(x_0)$ as any subgradient of $F(x)$ at $x_0$.

Recall that throughout the whole convergence analysis, we assume that Assumption~\ref{ass:main-problem1} and~\ref{ass:main-problem2} hold with all variables as defined in Algorithm~\ref{Alg:Cubic_PQN}.

\subsection{Preparatory lemmas}

First, we need several important properties of  $J_k\coloneqq\int_0^1\nabla^2f(x_k+tu_k)dt$ on which we will capitalize for the analysis of both algorithms.
\begin{lemma}\label{lemma:taylor}
    For each $k\in\N$, we have
    \begin{gather}
        f(x_\kp)-f(x_k)\leq \scal{\nabla f(x_k)}{u_k}  +  {\frac{1}{2}}\scal{ \nabla^2 f(x_k)u_k}{u_k}+\frac{L_H}{6}\norm[]{u_k}^3\,,\\
        f(x_\kp)-f(x_k)\leq \scal{\nabla f(x_k)}{u_k}  +  {\frac{1}{2}}\scal{ J_k u_k}{u_k}+\frac{L_H}{6}\norm[]{u_k}^3\,.
    \end{gather}
\end{lemma}
\begin {proof}
    The first result is from \cite[Lemma 1]{nesterov2006cubic}. Using a similar argument as in \cite[Lemma 1]{nesterov2006cubic}, we can derive the second result.
    For any $y,x\in\R^n$ and $t\in(0,1]$, we have
    \begin{equation}
        \begin{split}
            &\norm[]{\nabla f(x+t(y-x))-\nabla f(x) - \int_0^1\nabla^2 f(x+s(y-x))t(y-x)ds}\\&\leq \norm[]{\int_0^1(\nabla^2 f(x+st(y-x))- \nabla^2 f(x+s(y-x)))t(y-x)ds)}\\
            &\leq L_H\norm[]{y-x}^2
            \int_0^1\vert (s-st)t \vert ds\\
            &\leq \frac{(1-t)t}{2}L_H\norm[]{y-x}^2\,.
        \end{split}
    \end{equation}
    Therefore, we have
    \begin{equation}
    \begin{split}
        \vert f(x_\kp)-f(x_k)-\scal{\nabla f(x_k)}{u_k}- \frac{1}{2}\scal{ J_k u_k}{u_k}\vert &\leq \vert\int_0^1\scal{\nabla f(x_k+tu_k)-\nabla f(x_k) -tJ_ku_k}{u_k}dt \vert   \\
        &\leq \vert\vert u_k \vert\vert\int_0^1\vert\vert\nabla f(x_k+tu_k)-\nabla f(x_k) -tJ_ku_k\vert\vert dt   \\
        &\leq \frac{L_H}{2}\norm[]{u_k}^3\int_0^1 (t-t^2) dt\\
        &\leq \frac{L_H}{6}\norm[]{u_k}^3\,.\\
    \end{split}
    \end{equation}
    
\end {proof}

\begin{lemma}\label{lemma:bound Jk}
    For each $k\in\N$, we have $-L\opid\preceq J_k\preceq L\opid$.
\end{lemma}
\begin {proof}
  Since $f$ has $L$-Lipschitz gradient, we have $-L\opid\preceq\nabla^2 f(x)\preceq L\opid$ for any $x\in\R^n$. Thus, due to the definition of $J_k$, we have $-L\opid\preceq J_k\preceq L\opid$.
\end {proof}
\begin{lemma}\label{lemma:JleqJ}
For each $k\in\N$, we have
    \begin{gather}
        J_k\preceq J_\km + \frac{1}{2}L_H(r_k+r_\km)\opid\,,\\
        \nabla^2 f(x_k)\preceq J_{k-1} + \frac{1}{2}L_Hr_\km\opid\,,\\
        \label{ineq:H-Jk}
        \nabla^2 f(x_k)\preceq J_{k} + \frac{1}{2}L_Hr_k\opid\,.
    \end{gather}
    % and
    % \begin{equation}
    %     \frac{1}{(1+ \delta_k) }\tilde J_\km\preceq\tilde J_k\preceq (1+ \delta_k)\tilde J_\km\,,
    % \end{equation}
\end{lemma}
\begin {proof}
    The proof is similar to the one of \cite[Lemma 4.2]{rodomanov2021greedy}, \cite[Lemma 5]{ye2023towards}.  The assumption that the Hessian of $f$ is $L_H$-Lipschitz continuous means that there exists $L_H>0$ such that for any $x,y\in\R^n$, we have
\begin{equation}\label{ineq:LH}
    \norm[]{\nabla^2 f(x) - \nabla^2 f(y)}\leq L_H\norm[]{x-y}\,.
\end{equation}
For any $t,s\in[0,1]$ and $k\in\N$, letting $x=x_k+t u_k$ and $y=x_k-(1-s)u_\km$ in \eqref{ineq:LH}, we have that 
\begin{equation}\label{ineq:hessian}
    \norm[]{\nabla^2 f(x_k+tu_k) -\nabla^2 f(x_k-(1-s)u_k)}\leq L_H\norm[]{tu_k+(1-s)u_\km} \,.
    \end{equation}
Then, by the definition of $J_k$, we have
    \begin{equation}
        \begin{split}
            J_k-J_\km&=\int_0^1\nabla^2f(x_k+tu_k)dt-\int_0^1\nabla^2f(x_\km+su_\km)ds\\
            &=\int_0^1\nabla^2f(x_k+tu_k)dt-\int_0^1\nabla^2f(x_k-(1-s)u_\km)ds\\
            &\preceq \int_0^1\norm[]{ \nabla^2f(x_k+s u_k)-\nabla^2f(x_k-(1-s)u_\km)}ds \cdot\opid\\
            &\preceq\int_0^1L_H\norm[]{su_k+(1-s)u_\km}dt\cdot\opid\\
            &\preceq \frac{1}{2}L_H(r_k+r_\km)\opid\,,
        \end{split}
    \end{equation}
    where the first inequality holds because of the definition of  the induced matrix norm, the second inequality holds due to \eqref{ineq:hessian} and the last one uses the triangle inequality.
    Using the same trick, we have
    \begin{equation}
        \begin{split}
            \nabla^2 f(x_k)-J_\km&=\int_0^1\nabla^2f(x_k)dt-\int_0^1\nabla^2f(x_\km+su_\km)ds\\
            &\preceq \int_0^1\norm[]{ \nabla^2f(x_k)-\nabla^2f(x_k-(1-s)u_\km)}ds \cdot\opid\\
            &\preceq\int_0^1L_H\norm[]{(1-s)u_\km}dt\cdot\opid\\
            &\preceq \frac{1}{2}L_Hr_\km\opid\,.
        \end{split}
    \end{equation}
    We argue similarly to obtain \eqref{ineq:H-Jk}.
\end {proof}
We now recall an important property of the SR1 method, which shows that it preserves the partial order of matrices.
\begin{lemma}\label{lemma:Anton}
    For any two  {symmetric} matrices $A\in\R^{n\times n}$ and $G\in\R^{n\times n}$ with  {$A\preceq G$} and any $u\in\R^n$, we have the following for $G_+=\text{SR1}(A,G,u)$:
    \begin{equation}
        A\preceq G_+\preceq G\,.
    \end{equation}
\end{lemma}
\begin {proof}
     {The argument is similar to the one in \cite[Part of Lemma 2.2]{rodomanov2021greedy} despite they assume $A$ and $G$ are positive definite. We notice that the result holds true for general symmetric matrices $A$ and $G$ with $A\preceq G$. \textcolor{black}{Let $u$ be given.}
    If $(G-A)u= 0$, $G_+=G$. We now turn to the case when $ (G-A)u\neq 0$.  This means, since $G-A\succeq 0$, that $u^\top(G-A)u > 0$ for given $u\in\R^n$, and there exists a decomposition $G-A=VV^\top$ for some $V$.
    Then, we have
    \begin{equation}\label{eq:anton}
    \begin{split}
        G_+-A &= G-A - \frac{ (G-A)uu^\top (G-A)}{u^\top(G-A)u}\\
        & = VV^\top - V\frac{\tilde u\tilde u^\top}{\tilde u^\top\tilde u}V^\top\\
        &= V(\opid - \frac{\tilde u\tilde u^\top}{\tilde u^\top\tilde u})V^\top\succeq 0\,,\\
    \end{split}
    \end{equation}
     where $\tilde u=V^\top u$. Recall that since $u^\top(G-A)u>0$, $\tilde{u} \neq 0$ for the  given $u \neq 0$, and thus \eqref{eq:anton} makes sense.  Thus, $G_+\succeq A$. A similar argument also proves that $G_+\preceq G$. }
\end {proof}
Following \cite{wang2024global}, for any symmetric matrix $G$, we introduce the potential function:
\begin{equation}
    V(G)\coloneqq \trace G\,,
\end{equation}
and the following function:
\begin{equation}
    \nu(A, G,u) = \begin{cases}
        0\,,&\ \text{if $(G-A)u=0$}\,;\\
        \frac{u^\top( G-A)(G-A)u}{u^\top (G-A) u}\,,&\ \text{otherwise}\,.
    \end{cases}
    \end{equation}
%$\nu(A,G,u)$ measures the distance between $A$ and $G$. 
Using the potential function $V(G)$ and $\nu(A,G,u)$, the following lemma shows that the SR1 update leads to a better approximation of $A$.
\begin{lemma}\label{lemma:potentialfunction}
Consider two  {symmetric matrices} $A\in\R^{n\times n}$ and $ G\in\R^{n\times n}$ with $A\preceq  G$ and any $u\in\R^n$. Let $G_+=\text{SR1}(A,  G, u)$. Then, we have 
 \begin{equation}
    V(G) - V(G_+)=\nu(A, G,u)\,,
    \end{equation}
% where 
% \begin{equation}
%     \nu(A, G,u) =  \frac{\norm[]{(G-A)u}^2}{u^\top (G-A) u} \,.\\
%     \end{equation}
\end{lemma}
\begin {proof}
 {The argument is the same as the one in  \cite[Lemma~2.4]{wang2024global} where they assumed $A$ and $G$ are positive definite. However, the argument holds true for  symmetric matrices $A$ and $G$.}

% \color{gray}
% \begin{enumerate}
% \item If $( G-A)u=0$, it is obvious. If $( G-A)\neq 0$, we have
% \begin{equation}
%     \begin{split}
%         V(G)-V(G_+)&=\trace(G-G_+)= \trace\left(\frac{ (G-A)uu^\top(G- A)^\top }{u^\top(G-A)u}\right)= \left( \frac{\norm[]{(G-A)u}^2}{u^\top(G-A)u} \right)\\
%     \end{split}
% \end{equation}

%     \item If $( G-A)u=0$, it is obvious. We assume $( G-A)u\neq 0$. 
%     \begin{equation}
%     \begin{split}
%         V(G)-V(G_+)&=\trace(G-G_+)= \trace\left(\frac{ (G-A)uu^\top(G- A)^\top }{u^\top(G-A)u}\right)\geq \left( \frac{\norm[]{(G-A)u}^2}{u^\top G u} \right)\,/.\\
%     \end{split}
% \end{equation}
% where the last inequality holds since $G\succeq A\succ 0$.
%     \end{enumerate}
\end {proof}

\subsection{Convergence analysis of Algorithm~\ref{Alg:Cubic_PQN}}\label{conv:alg1}

Now, let us analyze Algorithm~\ref{Alg:Cubic_PQN}. Since Step~\ref{alg:SR1_da_PQN-step1} has two cases, we have to investigate the update step in \eqref{alg1:update} and the one in \eqref{alg1:update2} separately.
The necessary optimality condition of the update step in \eqref{alg1:update}, when $\trace G_k\leq n\bar\kappa$, implies
\begin{equation} \label{optimality conditioncase1:SR1}   
\nabla f(x_k) + \partial g(x_\kp)+ \left(G_k+L_H(r_{k-1}+r_k)\opid \right)(x_\kp-x_k)\ni0\,.
\end{equation}
In this case, we set $ {\tilde G_\kp}= G_k+\lambda_k\opid$ and $\lambda_k=L_H(r_{k-1}+r_k)$.
The necessary optimality condition of the update step in \eqref{alg1:update2}, when $\trace G_k> n\bar\kappa$, is
\begin{equation} \label{optimality conditioncase2:SR1}   
\nabla f(x_k) + \partial g(x_\kp)+ (L+L_Hr_{k-1}+L_Hr_k)(x_\kp-x_k)\ni0\,.
\end{equation}
In this case, we set $ {\tilde G_\kp}= (L+\lambda_k)\opid$.
In both cases, though the definitions of $\tilde G_k$ are different, \eqref{optimality conditioncase1:SR1} and \eqref{optimality conditioncase2:SR1} read:
\begin{equation} \label{optimality condition1:SR1}   
\nabla f(x_k) + v_\kp+   {\tilde G_\kp}(x_\kp-x_k)=0\,,
\end{equation}
where $v_\kp\in \partial g(x_\kp)$.

We denote $F^\prime(x_k)\coloneqq v_k+\nabla f(x_k)$. Thus, $F^\prime(x_k)\in\partial F(x_k)$.

\begin{lemma}\label{lemma:bound}
For each $k\in\N$, we have 
\begin{gather}\label{eq:lemma1:bound}
     J_k\preceq G_{k+1}\preceq  {\tilde G_\kp}\,,\\
    V (G_{k+1})\leq V  ({\tilde G_\kp})\,.
\end{gather}
\end{lemma}
\begin {proof}
We prove the inequalities in \eqref{eq:lemma1:bound} by induction. In order to validate the base case $k=0$, we first observe that $ {\tilde G_1}=G_0+\lambda_0\opid\succeq J_0$ holds by Lipschitz continuity of $\nabla f$ and $G_0=LI$ and $\lambda_0>0$. Then, Lemma~\ref{lemma:Anton} shows the desired property:
\[
J_0\preceq G_1=\text{SR1}(J_0, {\tilde G_1},u_0)\preceq {\tilde G_1}\,,
\]
and 
for showing the induction, we suppose now that \eqref{eq:lemma1:bound} holds for $k-1$. 
We discuss two cases separateluy:
\begin{enumerate}[1.]
        \item If $\trace G_{k}\leq n\bar\kappa$, then, $ {\tilde G_{k+1}}= G_k+\lambda_k\opid$.  According to the induction hypothesis and Lemma~\ref{lemma:JleqJ}, we have
\begin{equation}
    \begin{split}
          {\tilde G_{k+1}} &= G_k+\lambda_k\opid\\ 
        &\succeq J_\km + \lambda_k\opid\\&= J_\km + (L_Hr_\km + L_Hr_k)\opid\\&\succeq J_k\,.\\
    \end{split}
    \end{equation}
    We are then in position to apply Lemma~\ref{lemma:Anton} again. As a result, we obtain $J_k\preceq G_\kp=\text{SR1}(J_k, {\tilde G_{k+1}},u_k) \preceq  {\tilde G_{k+1}}$ and $\trace G_\kp\leq \trace  {\tilde G_{k+1}}$. 
    
\item If $\trace G_k>n\bar \kappa$, then $ {\tilde G_{k+1}}=(L+\lambda_k)\opid$ due to the restarting step.
    Using Lemma~\ref{lemma:bound Jk}, , we have $ {\tilde G_{k+1}}\succeq J_k$ and using Lemma~\ref{lemma:Anton} again, we obtain $J_k\preceq G_{k+1}\preceq  {\tilde G_{k+1}}$ and thus, $\trace G_{k+1}\leq \trace  {\tilde G_{k+1}}$. \qedhere
    \end{enumerate}
    % Then we discuss by cases to prove the rest two results:
    % \begin{enumerate}[1.]
    %     \item If $\trace \hat G_{N}\leq n\bar\kappa$, then, $\tilde G_N=\hat G_N$. 
    % %     According to Lemma~\ref{lemma:Anton} and $G_N=SR(J_{N-1},\tilde G_{N-1},u_{N-1})$, we obtain that 
    % % \begin{equation}
    % %     \tilde G_{N}=(1+\lambda_N)G_N\preceq (1+\lambda_N)\tilde G_{N-1}\preceq \left(\prod_{i=0}^{N}(1+\lambda_i)\right)L\opid\,.
    % % \end{equation}
    % % Thus, we obtain
    % \begin{equation}
    % \begin{split}
    %      (\text{Induction hypothesis})\ \tilde G_N &\succeq J_\Nm+\lambda_N\\ 
    %     &\succeq J_\Nm + 2L_Hr_\Nm + 2L_Hr_N\\(\text{Lemma~\ref{lemma:JleqJ}})&\succeq J_N\,.\\
    % \end{split}
    % \end{equation}
    % Using Lemma~\ref{lemma:Anton} again, we obtain $J_N\preceq G_\Np=SR(J_N,\tilde G_N,u_N) \preceq \tilde G_N$.
    % \item If $\trace \hat G_N>n\bar \kappa$, then $\tilde G_N=L\opid$.
    % Automatically, we have $\tilde G_N\succeq J_N$ and using Lemma~\ref{lemma:Anton} again, we obtain $G_{N+1}\succeq J_N$.
    % \end{enumerate}
\end {proof}

Based on the lemma above, we can show the monotone decrease of the function value.
\begin{lemma}\label{lemma1:descentVal}
    For each $k\in\N$, we have
    \begin{gather}
        F(x_\kp)-F(x_k)\leq -\frac{L_H}{6}r_k^3\,,\label{ineq:descentVal1}\\
        \scal{( {\tilde G_\kp}-J_k)u_k}{u_k}\leq 2(F(x_k)-F(x_\kp)) +\frac{2L_H}{3}r_k^3\,.\label{ineq:descentVal2}
    \end{gather}
\end{lemma}
\begin {proof}
    We prove again each case separately.
    If $\trace G_k\leq n\bar\kappa$,
    the optimality of $x_\kp$ in \eqref{alg1:update} implies that  
    \begin{equation}
         g(x_\kp)+\scal{\nabla f(x_k)}{u_k}+ \frac{1}{2}\scal{(G_k+L_Hr_\km)u_k}{u_k}+\frac{L_H}{3}r_k^3\leq g(x_k)\,.
    \end{equation}
    This together with Lemma 5.1 yields
    \begin{equation}
    \begin{split}
        F(x_\kp)&\leq F(x_k)-\frac{1}{2}\scal{(G_k+L_Hr_\km-\nabla^2 f(x_k))u_k}{u_k}-\frac{L_H}{6}r_k^3\\
        % & \leq F(x_k)-\frac{1}{2}\scal{(G_k+L_Hr_\km-J_\km-\frac{1}{2}L_Hr_\km)u_k}{u_k}-\frac{1}{6}r_k^3\\
        & \leq F(x_k)-\frac{1}{2}\scal{(G_k+L_Hr_\km-J_\km-\frac{1}{2}L_Hr_\km)u_k}{u_k}-\frac{L_H}{6}r_k^3\\
        & \leq F(x_k)-\frac{L_H}{6}r_k^3\,,\\
    \end{split}
    \end{equation}
    where the second inequality holds due to Lemma~\ref{lemma:JleqJ} and the last inequality holds due to Lemma~\ref{lemma:bound}.
    Similarly, thanks to Lemma~\ref{lemma:taylor}, we have
 \begin{equation}
    \begin{split}
        F(x_\kp)&\leq F(x_k)-\frac{1}{2}\scal{(G_k+L_Hr_\km-J_k)u_k}{u_k}-\frac{L_H}{6}r_k^3\\
        & = F(x_k)-\frac{1}{2}\scal{(G_k+L_Hr_\km+L_Hr_k-J_k-L_Hr_k)u_k}{u_k}-\frac{L_H}{6}r_k^3\\
        & \leq F(x_k)  -\frac{1}{2}\scal{( {\tilde G_\kp}-J_k)u_k}{u_k} +\frac{L_H}{3}r_k^3\,.\\
    \end{split}
    \end{equation}
     Let's now focus on the case where $\trace G_k> n\bar\kappa$. The update \eqref{alg1:update2} implies that  
    \begin{equation}
         g(x_\kp)+\scal{\nabla f(x_k)}{u_k}+ \frac{1}{2}\scal{(L+L_Hr_\km)u_k}{u_k}+\frac{L_H}{3}r_k^3\leq g(x_k)\,,
    \end{equation}
    since $x_\kp$ is the minimum of the subproblem.
    With Lemma~\ref{lemma:taylor}, we obtain
    \begin{equation}
    \begin{split}
        F(x_\kp)&\leq F(x_k)-\frac{1}{2}\scal{((L+L_Hr_\km)\opid-\nabla^2 f(x_k))u_k}{u_k}-\frac{L_H}{6}r_k^3\\
        & \leq F(x_k)-\frac{L_H}{6}r_k^3\,,\\
    \end{split}
    \end{equation}
    where the last inequality holds since $L\succeq \nabla^2 f(x_k)$. 
    Similarly, we have
 \begin{equation}
    \begin{split}
        F(x_\kp)&\leq F(x_k)-\frac{1}{2}\scal{((L+L_Hr_\km)\opid-J_k)u_k}{u_k}-\frac{L_H}{6}r_k^3\\
        & \leq F(x_k)-\frac{1}{2}\scal{((L+L_Hr_\km+L_Hr_k)\opid-J_k-L_Hr_k)u_k}{u_k}-\frac{L_H}{6}r_k^3\\
        & \leq F(x_k)  -\frac{1}{2}\scal{( {\tilde G_\kp}-J_k)u_k}{u_k} +\frac{L_H}{3}r_k^3\,.\\
    \end{split}
    \end{equation}
\end {proof}

 Due to the decrease of the function values, we have for any $k$, $r_k\leq \frac{F(x_0)-\inf F}{L_H}$. Thus, for any $k\in\N$, \textcolor{black}{$r_k\leq R$} where \textcolor{black}{$R=(6(F(x_0)-\inf F)/L_H)^{1/3}$}.
\begin{lemma}\label{lemma:boundedness of tilde G_k}
For each $k\in\N$, we have $\norm[]{ {\tilde G_\kp}}\leq 2n\bar\kappa+L+2L_H \textcolor{black}{R}$, where  \textcolor{black}{$R=(6(F(x_0)-\inf F)/L_H)^{1/3}$}..
\end{lemma}
\begin {proof}
    Case $\trace G_k\leq n\bar \kappa$: we have $ {\tilde G_\kp}=G_k+\lambda_k\opid=G_k+L_H(r_\km+r_k)\opid$. Thus, using Lemma~\ref{lemma:bound Jk} and Lemma~\ref{lemma:bound}, we get 
    $\norm[]{ {\tilde G_\kp}}\leq  \norm[]{\tilde G_\kp-J_\km}+\norm[]{J_\km}\leq  \trace(G_k-J_\km) + L+L_Hr_k+L_Hr_\km\leq n\bar\kappa+(n+1)L +L_Hr_k+L_Hr_\km\leq 2n\bar\kappa+L+2L_H\textcolor{black}{R}$ where  \textcolor{black}{$R=(6(F(x_0)-\inf F)/L_H)^{1/3}$}.
    
    Case $\trace G_k> n\bar \kappa$: we have $ {\tilde G_\kp}=L\opid+\lambda_k\opid$.
    Thus, we have $\norm[]{ {\tilde G_\kp}}\leq  L +L_Hr_k+L_Hr_\km\leq 2L+2L_HR\leq 2n\bar\kappa+L+2L_H\textcolor{black}{R}$.
\end {proof}
Let $\omega(x_0)$ be the set of clusters of $x_k$ generated by Algorithm~\ref{Alg:Cubic_PQN}.

\begin{lemma}\label{lemma:conv_x_k_to_Omega}
    \textcolor{black}{Assume $(x_k)_{k\in}$ generated by Algorithm~\ref{Alg:Cubic_PQN} is bounded.} The following holds:
    \begin{enumerate}[1.]
        \item 
        $\norm[]{F^\prime (x_k)}\to 0$ and $F(x_k)\to \bar{F}$ as $k\to\infty$,
        \item $\omega(x_0)\subset\mathrm{crit} F$,
        \item the objective function $F(x)\equiv \bar{F}$ is constant over $\omega(x_0)$ where $\bar{F}\in\R$, and
        %\item  $\mathrm{dist}(x_k,\Omega)$ as $k\to+\infty$. 
    \end{enumerate}
\end{lemma}
\begin {proof}
This proof is similar to that of \cite[Lemma 5]{bolte2014proximal}. 
    Thanks to Lemma~\ref{lemma1:descentVal} and $\inf F$ is finite (Assumption~\ref{ass:main-problem1}), we can deduce that $r_k\to0$ as $k\to \infty$ and $ F(x_k)\to \bar{F}$ for some $\bar{F}\in\R$.
    By \eqref{optimality condition1:SR1}, we have 
    \begin{equation}\label{eq:cubic_subseq_kq}
        F^\prime(x_\kp)=\nabla f(x_{k+1})-  \nabla f(x_{k})-  {\tilde G_\kp}(x_{k+1}-x_{k})\in \partial F(x_{k+1})\,.
    \end{equation}
    Since $\nabla f$ is Lipschitz continuous, $r_k\to 0$ as $k\to+\infty$ and $\tilde G_k$ is bounded, we have $\norm[]{F^\prime(x_\kp)}\to 0$ as $k\to+\infty$, , proving the first part of 1.
    Since the sequence $(x_k)_{k\in\N}$ is bounded, there exists a subsequence $(x_{k_q})_{q\in\N}$ that converges to some $\bar{x}$. 
    Since \eqref{eq:cubic_subseq_kq} holds true for any subsequences,  we get $\norm[]{F^\prime (x_{k_q})}\to 0$ as $q\to \infty$.
    % Thanks to the proper, lsc and convex function $F$, the graph of $\partial F$ is closed. 
    Letting $q\to\infty$, by the closedness of the graph of $\partial F$, we deduce that $\partial F(\bar x)\ni 0$. \textcolor{black}{For arbitrary $\bar x\in \omega(x_0)$, there exists a convergent subsequence $(x_{k_q})_{q\in\N}$ and $\partial F(\bar x)\ni 0$.} This shows 2.
    Since $x_{k_q}\to \bar x$ as $q\to\infty$ and $F(x_k)\to \bar{F}$ as $k\to \infty$, $F(x_{k_q})\to F(\bar x)=\bar{F}$ as $q\to\infty$. We recall that $\bar x$ is an arbitrary cluster. Therefore, $F$ is constant over $\omega(x_0)$. 
    % By the definition of $\Omega$, we have $\mathrm{dist}(x_k,\Omega)\to 0$ as $k\to\infty$.
\end {proof}

\begin{lemma}\label{lemma:subgradientbound}
    For any $k\in\N$, we have $\norm[]{F^\prime(x_\kp)}\leq Dr_k$ where $D= (2n\bar\kappa+2L+2L_H\textcolor{black}{R})$ and \textcolor{black}{$R=(6(F(x_0)-\inf F)/L_H)^{1/3}$}.
\end{lemma}
\begin {proof}
    By Lemma~\ref{lemma:boundedness of tilde G_k}, we have
    \begin{equation}
        \begin{split}
            \norm[]{F^\prime (x_\kp)}&= \norm[]{v_\kp+\nabla f(x_\kp)}\\
            &=\norm[]{\nabla f(x_\kp)-\nabla f(x_k)- {\tilde G_\kp}(x_\kp-x_k)}\\
            &\leq \norm[]{\nabla f(x_\kp)-\nabla f(x_k)}+\norm[]{ {\tilde G_\kp}(x_\kp-x_k)}\\
            &\leq Lr_k + (2n\bar\kappa+L+\textcolor{black}{2L_HR})r_k\\
            &\leq (2n\bar\kappa+2L+\textcolor{black}{2L_HR})r_k\,.\\
        \end{split}
    \end{equation}
\end {proof}
Next, we study the growth rates of the sequences $r_k$ and $\lambda_k$.
\begin{lemma}\label{lemma2:growthrk}
Given a number of iteration $N\in\N$, for large enough $k_0$, we have:
    \begin{gather}
        \sum_{k=k_0}^{N-1+k_0} r_{k}\leq C_0N^{1/2}\,,\\
        % \sum_{k=k_0+1}^{N+k_0} r_{k}\leq C_0N^{1/2} \,,
        % \\
        \sum_{k=k_0+1}^{N+k_0} \lambda_{k}\leq 2C_0L_H N^{1/2}\,,
    \end{gather}
    where $C_0\coloneqq\sqrt{ \frac{3r_{k_0}^{2}}{2}+ M(\phi(F(x_{k_0})-\textcolor{black}{\bar{F}}))}$ and $M\coloneqq\frac{3D}{L_H}$.
\end{lemma}
\begin {proof}
Thanks to Assumption~\ref{ass:main-problem1}, ~\ref{ass:main-problem2} and Lemma~\ref{lemma:uni_KL}, $F$ has the uniformized K\L\ property over \textcolor{black}{$\omega(x_0)$}, i.e., 
there exists $\epsilon>0$, $\eta>0$ and $\phi\in\Phi_\eta$ such that for all $\bar x\in\omega(x_0)$ and all $x$ in the following intersection
    \begin{equation}
        \set{x\in\R^n\vert\mathrm{dist}(x,\Omega)<\epsilon}\cap [ F(\bar x)<F< F(\bar x) +\eta]\,,
    \end{equation}
    one has, 
    \begin{equation}
        \phi^\prime \left(F(x)-F(\bar x)\right)\mathrm{dist}\left(0,\partial F(x)\right)\geq 1\,.
    \end{equation}
Moreover, in view of Lemma~\ref{lemma:conv_x_k_to_Omega}, there exists some $k_0$ such that for any $k\geq k_0$, $$x_k\in\set{x\in\R^n\vert\mathrm{dist}(x,\omega(x_0))<\epsilon}\cap [\bar{F}<F<\bar{F} +\eta]\,,$$
where $\bar{F}=F(\bar x)$.

We start with the first two inequalities. Using Lemma~\ref{lemma1:descentVal}, we have:
\begin{equation}\label{lemma:bound-kl}
\begin{split}
     \phi(F(x_k)-\textcolor{black}{\bar{F}})-\phi(F(x_\kp)-\textcolor{black}{\bar{F}})&\geq \phi^\prime(F(x_k)-\textcolor{black}{\bar{F}})(F(x_k)-F(x_\kp))\\
     &\geq \frac{ F(x_k)-F(x_\kp) }{ \mathrm{dist}(0,\partial F(x_k))}\geq \frac{ F(x_k)-F(x_\kp) }{\norm[]{F^\prime(x_k)}}\\
     &\geq \frac{ \frac{1}{  6}L_Hr_k^3}{\norm[]{F^\prime(x_k)}}\\
     %\textcolor{red}{\geq \frac{L_H r_k^2 }{2\lambda_k}\geq \frac{L_H \lambda^4_\kp }{2D^4\lambda_k}\,,}
     &\geq \frac{ \frac{1}{  6}L_Hr_k^3}{Dr_\km}\\
\end{split}
\end{equation}
where the first inequality uses the concavity of $\phi$, the second one uses the uniformized K\L\ inequality, the third one uses \eqref{eq:cubic_subseq_kq}, the fourth Lemma~\ref{lemma1:descentVal} and the last inequality uses Lemma~\ref{lemma:subgradientbound}. For convenience, as in \cite{bolte2014proximal}, we denote $\Delta_k \coloneqq \phi(F(x_k)-\textcolor{black}{\bar{F}})-\phi(F(x_\kp)-\textcolor{black}{\bar{F}})$ and we have for any $N$, 
\begin{equation}
    \sum_{k=k_0}^{N+k_0}\Delta_k\leq \phi(F(x_{k_0})-\textcolor{black}{\bar{F}})<+\infty\,.
\end{equation}
Then, we can derive from \eqref{lemma:bound-kl} that
\begin{equation}
    r_k^3\leq \frac{  6 D}{L_H}\Delta_kr_\km\,.
\end{equation}
Let $M \coloneqq \frac{  6 D}{L_H}$. We derive the following inequality by computing $2/3$-th order roots on both sides of the above inequality:
\begin{equation}\label{lemma:bound-lambda}
\begin{split}
    r_k^2 &\leq \left(\left(M\Delta_k\right)^2r_\km^{2}\right)^{1/3}\\
    &\leq \frac{r_\km^{2}}{3} + \frac{M\Delta_k}{3}+\frac{M\Delta_k}{3}\\
     &\leq \frac{r_\km^{2}}{3} + \frac{2M\Delta_k}{3}\\
\end{split}
\end{equation}
where the second inequality holds due to the AM-GM inequality $\frac{a+b+c}{3}\geq (abc)^{1/3}$. Using Lemma~\ref{app:lemma:combettes}, \eqref{lemma:bound-lambda} implies
% \begin{equation}
%     \sum_{k=k_0}^{N+k_0} r_k^{2}\leq \frac{3}{2}\left(r_{k_0}^{2} + \frac{2}{3}M\sum_{k=k_0}^{N+k_0} \Delta_k\right)\leq \frac{3r_{k_0}^{2}}{2}+ M(\phi(F(x_{k_0})-\textcolor{black}{\bar{F}}))<+\infty\,.
% \end{equation}
% Since the above inequality holds for any $N\in\N$, we let $N\to\infty$ and obtain 
\begin{equation}
    \sum_{k=k_0}^\infty r_k^{2}\leq \frac{3r_{k_0}^{2}}{2}+ M(\phi(F(x_{k_0})-\textcolor{black}{\bar{F}}))<+\infty\,.
\end{equation}

Using Cauchy-Schwarz inequality, for any $N$, we have 
\begin{gather}
    \sum_{k=k_0}^{N-1+k_0}r_k^{}\leq \left(\sum_{k=k_0}^{N-1+k_0}r_k^{2}\right)^{1/2}\left(\sum_{k=k_0}^{N-1+k_0} 1^{2}\right)^{1/2}\leq C_0N^{1/2}
    %\\
     % \sum_{k=1+k_0}^{N+k_0}r_k\leq \left(\sum_{k=1+k_0}^{N+k_0} r_k^{2}\right)^{1/2}\left(\sum_{k=1+k_0}^{N+k_0} 1^{2}\right)^{1/2}\leq C_0 N^{1/2}\,,  
\end{gather}
where $C_0\coloneqq \sqrt{ \frac{3r_{k_0}^{2}}{2}+ M(\phi(F(x_{k_0})-\textcolor{black}{\bar{F}}))}$.
By definition of $\lambda_k$, we can derive that
\begin{equation}
    \sum_{k=1+k_0}^{N+k_0}\lambda_k\leq L_H\sum_{1+k_0}^{N+k_0}(r_k+r_\km)\leq 2L_HC_0N^{1/2}\,.
\end{equation}
\end {proof}
 The following inequality will be crucial for achieving non-asymptotic convergence rates.
\begin{lemma}\label{grad:ukGu}
For any $k\geq k_0$, we have 
\begin{equation}\label{grad:ukGu:general_kl}
    \scal{u_k^\top}{  ({\tilde G_\kp}-J_k) u_k}\leq C_1 \norm[]{ F^\prime(x_k)}\,,
\end{equation}
where $C_1\coloneqq   6\phi(F(x_{k_0})-\textcolor{black}{\bar{F}})$.
\newline
Furthermore, if $\phi(t)=ct^{1-\theta}$ with $c>0$ and $\theta\in(0,1]$,
\begin{equation}\label{grad:ukGu:theta_kl}
    \scal{u_k^\top}{  ({\tilde G_\kp}-J_k) u_k}\leq   C_2 \norm[]{ F^\prime(x_k)}^{1/\theta}\,,
\end{equation}
where $C_2\coloneqq   6 ((1-\theta)c)^{1/\theta} $.
\end{lemma}
\begin {proof}
Let $k_0$ be the same one in Lemma~\ref{lemma2:growthrk}. 
We start with the first result for general $\phi(t)$.
% Due to Lemma~\ref{lemma1:descentVal}, we have
% \begin{equation}\label{grad:ukGu;ineq:4}
% \begin{split}
%     \scal{( {\tilde G_\kp}- J_k) u_k}{u_k}
%     &\leq  2(F(x_k)-F(x_\kp))+\frac{2L_H}{  3}r_k^3\,.
% \end{split}
% \end{equation}

From the K\L\ inequality and concavity of $\phi$, we have 
\begin{equation}\label{ineq:kl-abstract}
\begin{split}
    1&\leq \phi^\prime(F(x_k)-\bar{F})\mathrm{dist}{(0,\partial F(x_k))}\\
    &\leq \frac{\phi(F(x_k)-\bar{F}) -\phi(F(x_\kp)-\bar{F})}{F(x_k)-F(x_\kp)}\mathrm{dist}{(0,\partial F(x_k))}\\
    &\leq \frac{\phi(F(x_k)-\bar{F})}{F(x_k)-F(x_\kp)}\norm[]{ F^\prime(x_k)}\,.\\
\end{split}
\end{equation}
By applying \eqref{ineq:kl-abstract} along with the monotonicity of $\phi$ and $F(x_k)-\bar F$, we deduce that 
\begin{equation}\label{grad:ukGu;ineq:5}
    F(x_k)-F(x_\kp)\leq \phi(F(x_{k_0})-\textcolor{black}{\bar{F}}) \norm[]{ F^\prime(x_k)}\,.
\end{equation}
Combining \eqref{ineq:descentVal1} and \eqref{ineq:descentVal2}, we obtain the first desired inequality:
\begin{equation}\label{grad:ukGu;result1}
    \begin{split}
    \scal{( {\tilde G_\kp}- J_k) u_k}{u_k}
    &\leq  2(F(x_k)-F(x_\kp))+ {\frac{2L_H}{3}}r_k^3\,,\\
    &\leq   6(F(x_k)-F(x_\kp))\,\\
    &\leq   6\phi(F(x_{k_0})-\textcolor{black}{\bar{F}}) \norm[]{ F^\prime(x_k)}\,.\\
\end{split}
\end{equation}

Similarly, we prove the second result.
For the \L ojasiewicz case where $\phi(t)=ct^{1-\theta}$ and $\theta\in (0,1)$, we have 
\begin{equation}
\begin{split}
    1&\leq \phi^\prime(F(x_k)-\textcolor{black}{\bar{F}})\mathrm{dist}{(0,\partial F(x_k))}\\
    &=  (1-\theta)c\frac{1}{(F(x_k)-\textcolor{black}{\bar{F}})^{\theta}}\mathrm{dist}{(0,\partial F(x_k))}\\
    &\leq ((1-\theta)c)\frac{1}{(F(x_k)-\textcolor{black}{\bar{F}})^{\theta}}\norm[]{ F^\prime(x_k)}\,.\\
\end{split}
\end{equation}
Thus, we deduce that 
\begin{equation}\label{grad:ukGu;ineq:6}
    (F(x_k)-\textcolor{black}{\bar{F}})\leq ((1-\theta)c)^{1/\theta} \norm[]{ F^\prime(x_k)}^{1/\theta}\,.
\end{equation}
 {Using this instead of \eqref{grad:ukGu;ineq:5} in the last step of \eqref{grad:ukGu;result1}}, we get \eqref{grad:ukGu:theta_kl}.
\end {proof}
\begin{remark}
    In particular, when $\phi(t)=ct^{1/2}$ for some $c>0$, we have for $k\geq k_0$, 
    \begin{equation}
    \scal{u_k^\top}{  ({\tilde G_\kp}-J_k) u_k}\leq \frac{3c^2}{  2}\norm[]{F^\prime (x_k)}^2\,.\end{equation}
\end{remark}
\begin{lemma}\label{lemma2:maindescent}
For any $k\in\N$ and $k\geq k_0$, we have the inequality:
\begin{equation}\label{ineq:maindescent1}
    V( {\tilde G_\kp})-V( {\tilde G_{k+2}})\geq \frac{g^2_\kp}{C_1g_k}-n\lambda_\kp\,,
\end{equation}
where $g_k\coloneqq \norm[]{F^\prime(x_k)}$ and $C_1=  6\phi(F(x_{k_0})-\textcolor{black}{\bar{F}})$.
\newline
Furthermore, if we assume $\phi(t)=ct^{1-\theta}$, then, we have
\begin{equation}\label{ineq:maindescent2}
    V( {\tilde G_\kp})-V( {\tilde G_{k+2}})\geq \frac{g^2_\kp}{\min\{C_1g_k,C_2g_k^{1/\theta}\}}-n\lambda_\kp\,,
\end{equation}
where $C_1\coloneqq   6\phi(F(x_{k_0})-\textcolor{black}{\bar{F}})$ and $C_2 =  6 ((1-\theta)c)^{1/\theta} $.
\newline
In particular, if $\theta = 1/2$, we have
\begin{equation}
    V( {\tilde G_\kp})-V( {\tilde G_{k+2}})\geq \frac{  4 g^2_\kp}{3c^2g_k^2}-n\lambda_\kp\,.
\end{equation}
\end{lemma}
\begin {proof}
By the optimality condition in \eqref{optimality condition1:SR1}, we know that $ {\tilde G_\kp} u_k+v_\kp= -\nabla f(x_k)$ and by definition of $J_k$, we have $J_ku_k=\nabla f(x_\kp)-\nabla f(x_k)$. 
Using Lemma~\ref{lemma:potentialfunction}, we obtain:
    \begin{equation}\label{ineq2:V3}
    \begin{split}
         \textcolor{black}{V( {\tilde G_\kp})-V(G_\kp)=\nu(J_k,\tilde G_{k+1},u_k)} &= \frac{\norm[]{( {\tilde G_\kp}-J_k)u_k}^2}{u_k^\top ( {\tilde G_\kp}-J_k) u_k} = \frac{\norm[]{\nabla f(x_\kp)+v_\kp}^2}{u_k^\top ( {\tilde G_\kp}-J_k) u_k}\\
         &= \frac{\norm[]{F^\prime (x_\kp)}^2}{u_k^\top( {\tilde G_\kp}-J_k) u_k}\geq \frac{\norm[]{F^\prime (x_\kp)}^2}{C_1\norm[]{F^\prime(x_k)}}=\frac{ g_\kp^2}{C_1g_k}\,,
    \end{split}
    \end{equation}
    where the second inequality holds due to \text{Lemma~\ref{grad:ukGu}}. 
    For the case $\phi(t)=ct^{1/\theta}$, we derive from Lemma~\ref{grad:ukGu} that
    \begin{equation}\label{ineq2:V3theta}
    \begin{split}
         \textcolor{black}{V( {\tilde G_\kp})-V(G_\kp)=\nu(J_k, {\tilde G_\kp},u_k)}
         &= \frac{\norm[]{F^\prime (x_\kp)}^2}{u_k^\top( {\tilde G_\kp}-J_k) u_k}\\&\geq \frac{\norm[]{F^\prime (x_\kp)}^2}{\min\{C_1\norm[]{F^\prime(x_k)},C_2\norm[]{F^\prime(x_k)}^{1/\theta}\}}\\&=\frac{ g_\kp^2}{\min\{C_1g_k,C_2g_k^{1/\theta}\}}\,.\\
    \end{split}
    \end{equation}
    
    % Together with Lemma~\ref{lemma:potentialfunction} on the $k$-th iterate, we obtain
    % \begin{equation}\label{ineq2:V3}
    %     V(\tilde G_k) - V(G_\kp)=\textcolor{black}{\nu(J_k,\tilde G_k,u_k)}\geq\frac{\mu g_\kp^2}{g_k^2} \,.
    % \end{equation}
    We also need a lower bound for $V(G_\kp)-V( {\tilde G_{k+2}}) $, for which we need to consider the two cases of Step~\ref{alg:SR1_da_PQN-step1} separately.
    \begin{enumerate}[1.]
        \item When $\trace G_\kp\leq n\bar\kappa$, we have  {$\tilde G_{k+2}= G_\kp+\lambda_\kp\opid$}, and thus, 
        \begin{equation}\label{eq:lemma2:maindescent:proof1}
        V(G_\kp)-V(\tilde G_{k+2}) = \trace(G_\kp-G_\kp-\lambda_\kp\opid) = - n\lambda_\kp\,.
    \end{equation}
    \item When  {$\trace{ G_\kp}> n\bar\kappa$}, we have $V(  G_\kp)>n\bar\kappa\geq nL=\trace L \opid$. In this case we know that  {$\tilde G_{k+2}=(L+\lambda_\kp)\opid$}. Since $\lambda_k\geq0$, we deduce that
    \begin{equation}\label{eq:lemma2:maindescent:proof2}
    \begin{split}
        V(G_\kp)-V( {\tilde G_{k+2}})&= V(G_\kp)-V( {L\opid}) +V( {L\opid})-V(\tilde G_{k+2})\\&\geq V( {L\opid})-V( {\tilde G_{k+2}})  \\
       &\geq \trace( {L\opid}- {L\opid}-\lambda_\kp)\\
       &\geq -n\lambda_\kp\,.\\
    \end{split}
    \end{equation}
    \end{enumerate}
    Therefore, adding \eqref{ineq2:V3} (\eqref{ineq2:V3theta}) with \eqref{eq:lemma2:maindescent:proof1} ( \eqref{eq:lemma2:maindescent:proof2}) yields the desired inequality.
\end {proof}

Now, we are ready to prove our main theorem for Algorithm~\ref{Alg:Cubic_PQN}. 
\begin {proof}[\textbf{Proof of Theorem~\ref{thm:main-cubic}}]
For symplicity, we denote $\gamma_k^2 \coloneqq\frac{g^2_\kp}{a_k}$, where $a_k =C_1 g_k$ for general $\phi$ and $a_k ={\min\{C_1g_k,C_2g_k^{1/\theta}\}}$ for $\phi(t)=ct^{1-\theta}$.
    Summing the both sides of \eqref{ineq:maindescent1} and \eqref{ineq:maindescent2} from $ {k=k_0}$ to $ {N+k_0-1}$, we obtain 
    \begin{equation}
        V(\tilde G_{k_0+1})-V(\tilde G_{N+k_0+1})\geq\sum_{k=k_0}^{N-1+k_0} \gamma_k^2 - n\sum_{k=1+k_0}^{N+k_0} \lambda_k\,.
    \end{equation}
    Since, for every $k$, our method keeps $ {\tilde G_\kp}\succeq J_k$ and $G_\kp\succeq J_k$ (see Lemma~\ref{lemma:bound}),  we have $V(\tilde G_k)>-nL$ and therefore
    \begin{equation}\label{ineq2:sumineq}
    \begin{split}
    V( {\tilde G_{k_0+1}})+nL&\geq \sum_{k=k_0}^{N-1+k_0} \gamma_k^2 - n\sum_{k=1+k_0}^{N+k_0} \lambda_k\,.\\
    \end{split}
    \end{equation}
    By Lemma~\ref{lemma2:growthrk}, we obtain
    \begin{equation}\label{mainstep2}
    C_{\mathrm{CR1}}+C_{\mathrm{CR2}}N^{1/2}\geq V( {\tilde G_{k_0+1}})+nL+2nL_HC_0 N^{1/2}\geq \sum_{k=k_0}^{N-1+k_0}\gamma_k^2\,,
    \end{equation}
    where $C_{\mathrm{CR1}}\coloneqq (n+1)L+2n\bar\kappa+2L_HR\geq nL+ V( {\tilde G_{k_0+1}})$ and $C_{\mathrm{CR2}}\coloneqq 2nL_HC_0$. Dividing  \eqref{mainstep2} by $N$, we obtain 
    \begin{equation}\label{mainstep2.1}
    \frac{C_{\mathrm{CR1}}}{N}+\frac{C_{\mathrm{CR2}}}{N^{1/2}}\geq \frac{1}{N}\sum_{k=k_0}^{N-1+k_0}\gamma_k^2\,.
    \end{equation}
    We derive from \eqref{mainstep2.1} with the concavity and monotonicity of $\log x$ that
    \begin{equation}\label{ineq:key}
    \begin{split}
          \log\left(\frac{C_{\mathrm{CR1}}}{N}+\frac{C_{\mathrm{CR2}}}{N^{1/2}}\right) &\geq \log\left(\frac{1}{N}\sum_{k=k_0}^{N-1+k_0}\gamma_k^2\right)\geq\frac{1}{N}\sum_{k=k_0}^{N-1+k_0}\log(\gamma_k^2)= \log \left(\left(\prod_{k=k_0}^{N-1+k_0}\frac{g_\kp^2}{a_k}\right)^{1/N}\right)\\
          %=\log\left(\left(\frac{g_N}{g_0}\right)^{2/N}\right)\,.\\
    \end{split}
    \end{equation}
    We now discuss the two cases of Lemma~\ref{lemma2:maindescent} separately, depending on the desingularizing function $\phi$. For general $\phi$, we have $a_k= C_1 g_k$. \eqref{ineq:key} reads
    \begin{equation}
          \log\left(\frac{C_{\mathrm{CR1}}}{N}+\frac{C_{\mathrm{CR2}}}{N^{1/2}}\right)  \geq\log\left(\left(\frac{ g_{N+k_0}(\prod_{k=1+k_0}^{N+k_0} g_k) }{C^N_1g_{k_0}}\right)^{1/N}\right)\,.
    \end{equation}
    
    Thus, we obtain 
    \begin{equation}
         \left(\min_{\set{k\in\N\vert k_0\leq k\leq N+k_0}} g_k\right)^{\frac{N+1}{N}}\leq \left(g_{N+k_0}^{1/N}\right) \left(\prod_{k=1+k_0}^{N+k_0} g_k\right)^{1/N}\leq C_1\left(\frac{C_{\mathrm{CR1}}}{N}+\frac{C_{\mathrm{CR2}}}{N^{1/2}}\right) g_{k_0}^{1/N}\,.
    \end{equation}
    Taking the power $N/(N+1)$ on both sides shows the desired claim in \eqref{ineq:main-cubic1}:
    \begin{equation}
         \left(\min_{\set{k\vert k_0\leq k\leq N+k_0}} g_k\right)\leq \left(C_1 \left(\frac{C_{\mathrm{CR1}}}{N}+\frac{C_{\mathrm{CR2}}}{N^{1/2}}\right)\right)^{N/(N+1)}g_{k_0}^{2/(N+1)}\,.
    \end{equation}
    For $\phi(t)=ct^{1-\theta}$ with $\theta\in(0,\frac{1}{2}]$, we have $a_k ={\min\{C_1g_k,C_2g_k^{1/\theta}\}}$. 
    Let $A = C_1g_k$ and $B=C_2g_k^{1/\theta}$. We notice that by Jesen's inequality 
    $$
    \min\set{A,B}\leq A^{\frac{1-2\theta}{1-\theta}} B^{ \frac{\theta}{1-\theta}}\,.
    $$
    Then, $a_k\leq C g_k^2$ where $C = C_1^{ \frac{1-2\theta}{1-\theta}} C_2^{ \frac{\theta}{1-\theta}}$. In particular, when $\theta =\frac{1}{2}$, we have $C=\frac{3c^2}{4}$.

    By \eqref{ineq:key}, we have 
     \begin{equation}
          \log\left(\frac{C_{\mathrm{CR1}}}{N}+\frac{C_{\mathrm{CR2}}}{N^{1/2}}\right) \geq\log \left(\left(\prod_{k=k_0}^{N-1+k_0}\frac{g_\kp^2}{a_k}\right)^{1/N}\right)
          \geq \log\left(\frac{1}{C}\left(\frac{ g_{N+k_0} }{ g_{k_0}}\right)^{2/N}\right)\,.
    \end{equation}
    Thus, we obtain
     \begin{equation}
        g_{N+k_0}\leq\left(C\left(\frac{C_{\mathrm{CR1}}}{N}+\frac{C_{\mathrm{CR2}}}{N^{1/2}}\right)\right)^{N/2}g_{k_0}\,.
    \end{equation}
    For $\phi(t)=ct^{1-\theta}$ with $\theta\in(\frac{1}{2},1)$, we have $a_k \leq{\min\{C_1g_k,C_2g_k^{1/\theta}\}}\leq C_2g_k^{1/\theta} $.
    Thus, 
    \begin{equation}
          \log\left(\frac{C_{\mathrm{CR1}}}{N}+\frac{C_{\mathrm{CR2}}}{N^{1/2}}\right) \geq\log \left(\left(\prod_{k=k_0}^{N-1+k_0}\frac{g_\kp^2}{a_k}\right)^{1/N}\right)
          =\log\left(\left(\frac{1}{C_2}\right)\left(\frac{ g^2_N(\prod_{k=1+k_0}^{N-1+k_0} g_k^{2-\frac{1}{\theta}}) }{g_{k_0}^{1/\theta}}\right)^{1/N}\right)\,.
    \end{equation}
    Thus, we obtain

    \begin{equation}
        \min_{\set{k\vert k_0\leq k\leq N+k_0}} g_k\leq \left(C_2\left(\frac{C_{\mathrm{CR1}}}{N}+\frac{C_{\mathrm{CR2}}}{N^{1/2}}\right)\right)^{N/(2+(N-1)(2-\frac{1}{\theta}))}g_{k_0}^{1/(\theta(2+(N-1)(2-\frac{1}{\theta})))}\,.\qedhere
    \end{equation}
    
\end {proof}
\subsection{Convergence analysis of Algorithm~\ref{Alg:Grad_PQN}}\label{conv:alg2}
The necessary optimality condition of the update step in \eqref{alg_grad:update} implies 
\begin{equation} \label{optimality condition_grad:SR1}   
\nabla f(x_k) + v_\kp+ \tilde G_k(x_\kp-x_k)=0\,,
\end{equation}
where $v_\kp\coloneqq -\nabla f(x_k) -\tilde G_k(x_\kp-x_k) \in \partial g(x_\kp)$. Thus, $F^\prime(x_k)\in\partial F(x_k)$.
\begin{lemma}\label{lemma_grad:bound}
For each $k\in\N$, we have 
\begin{gather}\label{eq:lemma1_grad:bound}
     J_k\preceq G_{k+1}\preceq \tilde G_k\,,\\
    V(G_{k+1})\leq V( \tilde G_k)\leq n\bar\kappa\,.
\end{gather}
\end{lemma}
\begin {proof}
We first notice that since $g$ is convex, the monotonicity of its subdifferential yields \begin{equation}\label{ineq:innerpoduct positive}
    \scal{u_k}{v_\kp-v_k}\geq 0\,,
\end{equation} for any $k\in\N$. Thus, for any $k\in\N$, we have
\begin{equation}\label{grad ineq:ukF}
    \scal{u_k}{ \tilde G_k u_k}\leq \scal{u_k}{ \tilde G_k u_k} + \scal{u_k}{ v_\kp-v_k}=\scal{u_k}{-\nabla f(x_k) -v_\kp+v_\kp - v_k} =- \scal{u_k}{ F^\prime (x_k)} \,.
\end{equation}
The first inequality in \eqref{grad ineq:ukF} is a direct consequence of \eqref{ineq:innerpoduct positive}. 

We prove the remaining inequalities in \eqref{eq:lemma1_grad:bound} by induction. In order to validate the base case $k=0$, we first observe that $\tilde G_0=G_0+\lambda_0\opid\succeq J_0$ holds by Lipschitz continuity of $\nabla f$ and $G_0=LI$ and $\lambda_0=0$. Then, Lemma~\ref{lemma:Anton} shows the desired property:
\[
J_0\preceq G_1=\text{SR1}(J_0,\tilde G_0,u_0)\preceq\tilde G_0\preceq nL\opid\preceq n\bar\kappa\opid\,.
\]
For showing the induction, we suppose now that \eqref{eq:lemma1_grad:bound} holds for $k-1$. 
We discuss each case separately.
\begin{enumerate}[1.]
        \item If $\trace \hat G_{k}\leq n\bar\kappa$, then, $\tilde G_k=\hat G_k$.  Since $\tilde G_k=G_k+\lambda_k\opid\succeq \lambda_k\opid $, the above equality \eqref{grad ineq:ukF} implies
\begin{equation}
    \lambda_k \norm[]{u_k}^2 \leq u_k^\top (G_k+\lambda_k\opid) u_k= u_k^\top \tilde G_k u_k\leq-u_k^\top F^\prime(x_k)\leq \norm[]{u_k}\norm[]{F^\prime(x_k)}\,.
\end{equation}
Thus, $ \norm[]{u_k}\leq \frac{\norm[]{F^\prime(x_k)}}{\lambda_k}$. 
Therefore,
\begin{equation}
    L_Hr_k= L_H\norm[]{u_k}\leq \frac{L_H\norm[]{F^\prime(x_k)}}{\lambda_k}\leq \sqrt{L_H\norm[]{F^\prime(x_k)} }\,,
\end{equation}
where the last inequality holds since $\lambda_k\geq \sqrt{L_H\norm[]{F^\prime(x_k)} }$ (see the definition of $\lambda_k$).
We can deduce that:
\begin{equation}
     L_Hr_k + L_Hr_{k-1} \leq \sqrt{L_H\norm[]{F^\prime(x_k)}} + L_Hr_{k-1}=\lambda_k\,.
\end{equation}
Therefore, we deduce by the induction hypothesis and Lemma~\ref{lemma:JleqJ} that
\begin{equation}
    \begin{split}
         \ \tilde G_k &= G_k+\lambda_k\opid\\&\succeq J_\km + \lambda_k\opid\\&\succeq J_\km + (L_Hr_\km + L_Hr_k)\opid\\&\succeq J_k\,.\\
    \end{split}
    \end{equation}
    We are then in position to apply Lemma~\ref{lemma:Anton} again to get $J_k\preceq G_\kp=\text{SR1}(J_k,\tilde G_k,u_k) \preceq \tilde G_k$ and $\trace G_\kp\leq \trace \tilde G_k\leq n\bar\kappa$. 
    
\item If $\trace \hat G_N>n\bar \kappa$, then $\tilde G_N=L\opid$ due to the restarting step.
    Automatically, we have $\tilde G_N\succeq J_N$ and using Lemma~\ref{lemma:Anton} again, we obtain $J_N\preceq G_{N+1}\preceq \tilde G_N\preceq L\opid$ and $\trace G_\Np\leq \trace \tilde G_N=nL\leq n\bar\kappa$.\qedhere
    \end{enumerate}
\end {proof}
\begin{lemma}\label{lemma:nabla2_f_tilde_G}
    For any $k\in \N$ and $t\in[0,1]$, we have
    \begin{equation}
        \nabla^2 f(x_k+tu_k)\preceq \tilde G_k\,.
    \end{equation}
\end{lemma}
\begin {proof}
    Following the proof for Lemma \ref{lemma:JleqJ}, we have 
    \begin{equation}
        \nabla^2 f(x_k+tu_k)\preceq \textcolor{black}{J_\km+\frac{1}{2}L_Hr_\km} + \frac{1}{2}L_Hr_k\,.
    \end{equation}
    By Lemma~\ref{lemma_grad:bound}, we have $J_\km\preceq G_k$. Thus, for the case $\hat G_k\leq n\bar\kappa$, we have
    \begin{equation}
        \nabla^2 f(x_k+tu_k)\preceq J_\km+\frac{1}{2}L_Hr_\km + \frac{1}{2}L_Hr_k\preceq G_k+\frac{1}{2}L_Hr_\km+\frac{1}{2}L_Hr_k= \tilde G_k\,.
    \end{equation}
    Otherwise, $\tilde G_k = L\opid$. We still have
    \begin{equation}
        \nabla^2 f(x_k+tu_k)\preceq \tilde G_k\,.
    \end{equation}
\end {proof}
\begin{lemma}\label{lemma_grad:descentVal}
For any $k\in\N$, we have
\begin{equation}
    F(x_\kp)-F(x_k)\leq -\frac{1}{2}\scal{\tilde G_ku_k}{u_k}\,.
\end{equation}
\end{lemma}

\begin {proof}
    We start with the Taylor expansion of the smooth function $f$ and we obtain
    \begin{equation}\label{grad decentlemma ineq:1}
        \begin{split}
           f(x_\kp)-f(x_k)&= \scal{\nabla f(x_k)}{u_k}  + \int_0^1 (1-t)\scal{\nabla^2 f(x_k+tu_k)u_k}{u_k}dt\\
            &\leq \scal{\nabla f(x_k)}{u_k}+ \int_0^1 (1-t)\scal{\tilde G_ku_k}{u_k}dt \\
            &=\scal{\nabla f(x_k)}{u_k}+\frac{1}{2}\scal{\tilde G_ku_k}{u_k}\,,\\
        \end{split}
    \end{equation}
    where the first inequality holds due to Lemma~\ref{lemma:nabla2_f_tilde_G}. Since $x_\kp$ is the solution of the inclusion \eqref{optimality condition_grad:SR1} and $g$ is convex, the subgradient inequality holds:
    \begin{equation}\label{grad decentlemma ineq:4}
        g(x_\kp) -g(x_k)\leq \scal{v_\kp}{u_k}\,.
    \end{equation}
    Combining \eqref{grad decentlemma ineq:1} and \eqref{grad decentlemma ineq:4}, we deduce that 
    \begin{equation}
    \begin{split}
        F(x_\kp) -F(x_k)&\leq \scal{v_\kp+ \nabla f(x_k)}{u_k}+\frac{1}{2}\scal{\tilde G_ku_k}{u_k}\\
           &=-\frac{1}{2}\scal{\tilde G_ku_k}{u_k}\,,  \\
    \end{split}
    \end{equation}
    where we used \eqref{optimality condition_grad:SR1} to get the last equality.
\end {proof}
Due to the monotone decrease of the function values, we have for any $k$, $x_k\in [F\leq F(x_0)]$.

\begin{lemma}\label{lemma2:conv_x_k_to_Omega}
   \textcolor{black}{Assume $(x_k)_{k\in}$ generated by Algorithm~\ref{Alg:Grad_PQN} is bounded.} We have the following results:
    \begin{enumerate}[1.]
        \item $\norm[]{F^\prime (x_k)}\to 0$ and $F(x_k)\to \min F$,
        \item $\omega(x_0)\subset\mathrm{crit} F$,
        \item The objective function $F(x)\equiv \min F$ over $\Omega$ where $\min F$ is finite.
    \end{enumerate}
\end{lemma}
\begin {proof}
    The argument remains the same as that of Lemma~\ref{lemma:conv_x_k_to_Omega}.
\end {proof}
\subsubsection{Proof of Theorem~\ref{thm:main-grad}}
\begin{lemma}\label{lemma_grad:growthrk}
Given a number of iteration $N\in\N$ and $k\geq k_0$ for large enough $k_0$, we have:
    \begin{gather}
        \norm[]{F^\prime (x_k)}\leq \frac{\lambda_k^2}{L_H}\,,\\
        \lambda_\kp \leq D \sqrt{r_k}\,, \\
        \sum_{k=1+k_0}^{N+k_0} \lambda_k\leq DC_0^{1/3} N^{2/3}\,,
    \end{gather}
    where $C_0\coloneqq \frac{r_{k_0}^{3/2}}{3}+ M(\phi(F(x_{k_0})-\textcolor{black}{\inf F}))$, $M=\sqrt{\frac{4(n\bar\kappa+L)}{L_H}}$, $D \coloneqq (\sqrt{L_H  (n\bar\kappa +  L)} + L_HR)$ and $R\geq \max_{k\in \N} r_k$.
\end{lemma}
\begin {proof}
Thanks to Assumption~\ref{ass:main-problem1}, ~\ref{ass:main-problem2} and Lemma~\ref{lemma:uni_KL}, $F$ has uniformized K\L\ property with respect to \textcolor{black}{$\omega(x_0)$}. We recall that we assume $\argmin F\neq \emptyset$. Then, there exist $\epsilon>0$, $\eta>0$ and $\phi\in\Phi_\eta$ such that for all $\bar x\in\omega(x_0)$ and all $x$ in the following intersection
    \begin{equation}
        \set{x\in\R^n\vert\mathrm{dist}(x,\Omega)<\epsilon}\cap [ F(\bar x)<F< F(\bar x) +\eta]\,,
    \end{equation}
    one has, 
    \begin{equation}
        \phi^\prime \left(F(x)-F(\bar x)\right)\mathrm{dist}\left(0,\partial F(x)\right)\geq 1\,.
    \end{equation} 
    Thanks to Lemma~\ref{lemma_grad:descentVal} and~\ref{lemma2:conv_x_k_to_Omega}, there exists some $k_0$ such that for any $k\geq k_0$, $$x_k\in\set{x\in\R^n\vert\mathrm{dist}(x,\Omega)<\epsilon}\cap [ \min F<F(x)<\min F +\eta]\,.$$

We start with the first inequality. Step~\ref{alg:SR1_da_grad_PQN2-step3} implies that $\lambda_k\geq \sqrt{L_H\norm[]{F^\prime(x_k)}}$ since $r_k\geq 0$ for any $k\in\N$. Thus, we obtain
\begin{equation}\label{lemma_grad:bound-ineq1}
\norm[]{F^\prime (x_k)}\leq \frac{\lambda_k^2}{L_H}\,.    
\end{equation}

Now, we prove the second inequality. From the definition of $F^\prime(x_\kp)$, it follows that  
\begin{equation}\label{lemma_grad:bound-ineq2}
\begin{split}
    \norm[]{F^\prime(x_\kp)} &=  \norm[]{-\tilde G_k u_k -(\nabla f(x_k)-\nabla f(x_\kp))}\\
    &\leq   \norm[]{\tilde G_k u_k} +\norm[]{(\nabla f(x_k)-\nabla f(x_\kp))} \\
     &\leq   n\bar\kappa r_k +  L r_k \\
      &\leq   (n\bar\kappa +  L) r_k\,. \\
\end{split}
\end{equation}
Since the boundedness of $(x_k)_{k\in\N}$ is assumed, there exists some $\textcolor{black}{R}>0$ such that for any $k\in\N$, $r_k\leq \textcolor{black}{R}$.
According to step~\ref{alg:SR1_da_grad_PQN2-step3}, we have 
\begin{equation}\label{lemma_grad:bound-ineq3}
    \begin{split}
        \lambda_\kp &= \sqrt{L_H \norm[]{F^\prime(x_\kp)}} + L_Hr_k\\
        &\leq \sqrt{L_H  (n\bar\kappa +  L) r_k} + L_Hr_k\\
         &\leq D\sqrt{r_k}\,,\\
    \end{split}
\end{equation}
where $D= (\sqrt{L_H  (n\bar\kappa +  L)} + L_H\sqrt{\textcolor{black}{R}})$.

Finally, we are going to show the third inequality. According to Lemma~\ref{lemma_grad:descentVal}, we have 
\begin{equation}\label{lemma_grad:bound-ineq4}
    F(x_\kp) - F(x_k)\leq -\frac{1}{2}\scal{\tilde G_k u_k}{u_k}\leq -\frac{1}{2}\lambda_k r_k^2\,,
\end{equation}
 where the last inequality holds since $G_k$ is symmetric positive semi-definite and $\tilde G_k=G_k+\lambda_k\opid$. 
 
Since $F$ satisfies K\L\     inequality, we adopt the proof from \cite{bolte2014proximal}. Using \eqref{lemma_grad:bound-ineq1}, \eqref{lemma_grad:bound-ineq2} and \eqref{lemma_grad:bound-ineq4} , we obtain: for any $k\geq k_0$,
\begin{equation}\label{lemma_grad:bound-kl}
\begin{split}
     \phi(F(x_k)-\textcolor{black}{\inf F})-\phi(F(x_\kp)-\textcolor{black}{\inf F})&\geq \phi^\prime(F(x_k)-\textcolor{black}{\inf F})(F(x_k)-F(x_\kp))\\
     &\geq \frac{ (F(x_k)-F(x_\kp)) }{ \mathrm{dist}(0,\partial F(x_k))}\geq \frac{ (F(x_k)-F(x_\kp)) }{\norm[]{F^\prime(x_k)}}\\
     &\geq \frac{ \lambda_kr_k^2 }{2\norm[]{F^\prime(x_k)}}\\
     %\textcolor{red}{\geq \frac{L_H r_k^2 }{2\lambda_k}\geq \frac{L_H \lambda^4_\kp }{2D^4\lambda_k}\,,}
     &\geq \frac{ \lambda_kr_k^2 }{2\norm[]{F^\prime(x_k)}^{1/2} \norm[]{F^\prime(x_k)}^{1/2} }\\
     &\geq \frac{ \lambda_kr_k^2 }{2\sqrt{\frac{\lambda_k^2}{L_H}}\sqrt{(n\bar\kappa +L)r_\km} }\\
      &\geq \frac{ \sqrt{L_H}r_k^2 }{2\sqrt{(n\bar\kappa +L)r_\km} }\\
\end{split}
\end{equation}
where the first inequality uses the concavity of $\phi$. For convenience, as in \cite{bolte2014proximal}, we denote $\Delta_k \coloneqq \phi(F(x_k)-\textcolor{black}{\inf F})-\phi(F(x_\kp)-\textcolor{black}{\inf F})$ and we have for any $N\in\N$, 
\begin{equation}
    \sum_{k=k_0}^{N+k_0}\Delta_k\leq \phi(F(x_{k_0})-\textcolor{black}{\inf F})<+\infty\,.
\end{equation}
Then, we can derive from \eqref{lemma_grad:bound-kl} that
\begin{equation}
    r_k^2\leq \sqrt{\frac{4(n\bar\kappa+L)}{L_H}}\Delta_k\sqrt{r_\km}\,.
\end{equation}
Let $M$ denote $ \sqrt{\frac{4(n\bar\kappa+L)}{L_H}}$. We derive the following inequality by computing $3/4$-th order root on both sides of the above inequality:
\begin{equation}\label{lemma_grad:bound-lambda}
\begin{split}
    r_k^\frac{3}{2} &\leq \left(\left(M\Delta_k\right)^3r_\km^{3/2}\right)^{1/4}\\
    &\leq \frac{r_\km^{3/2}}{4} + \frac{M\Delta_k}{4}+\frac{M\Delta_k}{4}+\frac{M\Delta_k}{4}\\
     &\leq \frac{r_\km^{3/2}}{4} + \frac{3M\Delta_k}{4}\\
\end{split}
\end{equation}
where the second inequality holds due to the AM-GM inequality $\frac{a+b+c+d}{4}\geq (abcd)^{1/4}$.
Summing up \eqref{lemma_grad:bound-lambda} from $1+k_0$ to any $N+k_0$, by reorganizing, we have
\begin{equation}
    \sum_{k=1+k_0}^{N+k_0} r_k^{3/2}\leq \frac{r_{k_0}^{3/2}}{3} + M\sum_{k=1+k_0}^{N+k_0} \Delta_k\leq \frac{r_{k_0}^{3/2}}{3}+ M(\phi(F(x_{k_0})-\textcolor{black}{\inf F}))<+\infty\,.
\end{equation}
We denote $C_0= \frac{r_{k_0}^{3/2}}{3}+ M(\phi(F(x_{k_0})-\textcolor{black}{\inf F})) $. 
Using H\"older inequality, for any $N$, we have 
\begin{gather}
    \sum_{k=1+k_0}^{N+k_0}r_k^{1/2}\leq \left(\sum_{k=1+k_0}^{N+k_0} r_k^{3/2}\right)^{1/3}\left(\sum_{k=1+k_0}^{N+k_0} 1^{3/2}\right)^{2/3}\leq C_0^{1/3}N^{2/3}\\
     \sum_{k=1+k_0}^{N+k_0}r_k\leq \left(\sum_{k=1+k_0}^{N+k_0} r_k^{3/2}\right)^{2/3}\left(\sum_{k=1+k_0}^{N+k_0} 1^{3}\right)^{1/3}\leq C_0^{2/3}N^{1/3}\,.  
\end{gather}
From \eqref{lemma_grad:bound-ineq3}, we can derive that
\begin{equation}
    \sum_{k=1+k_0}^{N+k_0}\lambda_k\leq D\sum_{k=1+k_0}^{N+k_0}r_k^{1/2}\leq DC_0^{1/3}N^{2/3}\,.
\end{equation}
\end {proof}
\begin{lemma}\label{grad:ukGu2}
For any $k\geq k_0$ where $k_0$ is large enough, we have 
\begin{equation}
    \scal{u_k}{ \tilde G_k u_k}\leq S_1\norm[]{F^\prime (x_k)}\,,
\end{equation}
where $S_1\coloneqq 2\phi(F(x_{k_0})-\textcolor{black}{\inf F})$.
Furthermore, if $\phi(t)=ct^{1-\theta}$ with $c>0$,
\begin{equation}\label{grad:ukGu2_ineq2}
    \scal{u_k}{ \tilde G_k u_k}\leq \min\{S_2\norm[]{F^\prime (x_k)}^{1/\theta},  S_1\norm[]{F^\prime (x_k)} \}\,,
\end{equation}
where $S_2 =2((1-\theta)c)^{1/\theta}$.
\end{lemma}
\begin {proof}
Similar to the proof of Lemma~\ref{lemma_grad:growthrk}, 
there exists $\epsilon>0$, $\eta>0$ and $\phi\in\Phi_\eta$ such that for all $\bar x\in\omega(x_0)$ and all $x$ in the following intersection
    \begin{equation}
        \set{x\in\R^n\vert\mathrm{dist}(x,\Omega)<\epsilon}\cap [ F(\bar x)<F< F(\bar x) +\eta]\,,
    \end{equation}
    one has, 
    \begin{equation}
        \phi^\prime \left(F(x)-F(\bar x)\right)\mathrm{dist}\left(0,\partial F(x)\right)\geq 1\,.
    \end{equation}

Then, there exists some $k_0$ such that for any $k\geq k_0$, $$x_k\in\set{x\in\R^n\vert\mathrm{dist}(x,\Omega)<\epsilon}\cap [ \min F<F(x)<\min F +\eta]\,.$$
We start with the first result for general $\phi(t)$.
From the K\L\ inequality and concavity of $\phi$, we have for any $k\geq k_0$, 
\begin{equation}\label{ineq_grad:kl-abstract}
\begin{split}
    1&\leq \phi^\prime(F(x_k)-\inf F)\mathrm{dist}{(0,\partial F(x_k))}\\
    &\leq \frac{\phi(F(x_k)-\inf F) -\phi(F(x_\kp)-\inf F)}{F(x_k)-F(x_\kp)}\mathrm{dist}{(0,\partial F(x_k))}\\
    &\leq \frac{\phi(F(x_k)-\inf F)}{F(x_k)-F(x_\kp)}\norm[]{ F^\prime(x_k)}\,.\\
\end{split}
\end{equation}
From \eqref{ineq_grad:kl-abstract}, using Lemma~\ref{lemma_grad:descentVal}, we deduce that 
\begin{equation}
    (F(x_k)-F(x_\kp))\leq \phi(F(x_{k})-\textcolor{black}{\inf F}) \norm[]{ F^\prime(x_k)}\leq \phi(F(x_{k_0})-\textcolor{black}{\inf F}) \norm[]{ F^\prime(x_k)}\,,
\end{equation}
where the first inequality holds since $\phi$ is increasing and the sequence $F(x_k)$ is nonincreasing. 
According to Lemma~\ref{lemma_grad:descentVal}, we have 
\begin{equation}
    u_k^\top \tilde G_k u_k\leq 2(F(x_k)-F(x_\kp))\leq 2\phi(F(x_{k_0})-\textcolor{black}{\inf F})\norm[]{F^\prime (x_k)}\,.
\end{equation}
Similarly, we prove the second result.
For the \L ojasiewicz case where  $\phi(t)=ct^{1-\theta}$, we have 
\begin{equation}
\begin{split}
    1&\leq \phi^\prime(F(x_k)-\textcolor{black}{\inf F})\mathrm{dist}{(0,\partial F(x_k))}\\
    &=  (1-\theta)c\frac{1}{(F(x_k)-\textcolor{black}{\inf F})^{\theta}}\mathrm{dist}{(0,\partial F(x_k))}\\
    &\leq ((1-\theta)c)^{1/\theta}\frac{1}{(F(x_{k_0})-\textcolor{black}{\inf F})^{\theta}}\norm[]{ F^\prime(x_k)}\,.\\
\end{split}
\end{equation}
% Thus, we deduce that 
% \begin{equation}
%     (F(x_k)-\textcolor{black}{\inf F})\leq \frac{c^2}{4} \norm[]{ F^\prime(x_k)}^2\,.
% \end{equation}
According to Lemma~\ref{lemma_grad:descentVal}, we have 
\begin{equation}\label{ineq:kl-theta}
    u_k^\top \tilde G_k u_k\leq 2(F(x_k)-F(x_\kp))\leq 2(F(x_{k_0})-\textcolor{black}{\inf F})\leq 2((1-\theta)c)^{1/\theta}\norm[]{F^\prime (x_k)}^{1/\theta}\,.
\end{equation}
By combining \eqref{ineq_grad:kl-abstract} and \eqref{ineq:kl-theta}, we obtain \eqref{grad:ukGu2_ineq2}. 
\end {proof}
\begin{remark}
    In particular, when $\phi(t)=ct^{1/2}$ for some $c>0$, we have for any $k\geq k_0$, 
    \begin{equation}
    u_k^\top \tilde G_k u_k\leq \frac{c^2}{2}\norm[]{F^\prime (x_k)}^2\,.
\end{equation}
\end{remark}

\begin{lemma}\label{lemma_grad:maindescent}
For any $k\in\N$ and $k>k_0$, we have a descent inequality as the following:
\begin{equation}\label{ineq:lemma_grad:maindescent_1}
    V(\tilde G_k)-V(\tilde G_\kp)\geq \frac{g^2_\kp}{S_1g_k}-n\lambda_\kp\,,
\end{equation}
where $g_k\coloneqq \norm[]{F^\prime(x_k)}$ and $S_1=2\phi(F(x_{k_0})-\textcolor{black}{\inf F})$.
\newline
Furthermore, if we assume $\phi(t)=ct^{1-\theta}$, then, we have
\begin{equation}\label{ineq:lemma_grad:maindescent_2}
    V(\tilde G_k)-V(\tilde G_\kp)\geq \frac{g^2_\kp}{\min\{S_1g_k,S_2g_k^{1/\theta}\}}-n\lambda_\kp\,,
\end{equation}
where $S_2 =2((1-\theta)c)^{1/\theta}$.
\newline
In particular, if $\theta = 1/2$, we have
\begin{equation}
    V(\tilde G_k)-V(\tilde G_\kp)\geq \frac{2 g^2_\kp}{c^2g_k^2}-n\lambda_\kp\,.
\end{equation}
\end{lemma}
\begin {proof}
By the optimality condition in \eqref{optimality condition_grad:SR1}, we know that $\tilde G_k u_k+v_\kp= -\nabla f(x_k)$ and by definition of $J_k$, we have $J_ku_k=\nabla f(x_\kp)-\nabla f(x_k)$. 
Using Lemma~\ref{lemma:potentialfunction}, we obtain:
    \begin{equation}\label{ineq_grad:V3}
    \begin{split}
         \textcolor{black}{V(\tilde G_k)-V(G_\kp)=\nu(J_k,\tilde G_k,u_k)} &\geq \frac{\norm[]{(\tilde G_k-J_k)u_k}^2}{u_k^\top \tilde G_k u_k} = \frac{\norm[]{\nabla f(x_\kp)+v_\kp}^2}{u_k^\top \tilde G_k u_k}\\
         &= \frac{\norm[]{F^\prime (x_\kp)}^2}{u_k^\top\tilde G_k u_k}\geq \frac{\norm[]{F^\prime (x_\kp)}^2}{S_1\norm[]{F^\prime(x_k)}}=\frac{ g_\kp^2}{S_1g_k}\,,
    \end{split}
    \end{equation}
    where the second inequality holds due to \text{Lemma~\ref{grad:ukGu2}}. 
    For the case $\phi(t)=ct^{1/\theta}$, we derive from Lemma~\ref{grad:ukGu2} that
    \begin{equation}\label{ineq_grad:V3theta}
    \begin{split}
         \textcolor{black}{V(\tilde G_k)-V(G_\kp)=\nu(J_k,\tilde G_k,u_k)}
         &= \frac{\norm[]{F^\prime (x_\kp)}^2}{u_k^\top\tilde G_k u_k}\\&\geq \frac{\norm[]{F^\prime (x_\kp)}^2}{\min\{S_1\norm[]{F^\prime(x_k)},S_2\norm[]{F^\prime(x_k)}^{1/\theta}\}}\\&=\frac{ g_\kp^2}{\min\{S_1g_k,S_2g_k^{1/\theta}\}}\,,\\
    \end{split}
    \end{equation}

    We also need a lower bound for $V(G_\kp)-V(\tilde G_\kp) $, for which we need to consider the two cases of the restarting Step~\ref{alg:SR1_da_grad_PQN2-step4}.
    \begin{enumerate}[1.]
        \item When $\norm[2]{\hat G_\kp}\leq n\bar\kappa$, we have $\tilde G_\kp=\hat G_\kp$, and
        \begin{equation}\label{eq:lemma_grad:maindescent:proof1}
        V(G_\kp)-V(\tilde G_\kp) = \trace(G_\kp-G_\kp-\lambda_\kp ) = - n\lambda_\kp\,.
    \end{equation}
    \item When $\trace{\hat G_\kp}> n\bar\kappa$, we have $V(\hat G_\kp)>n\bar\kappa\geq nL=\trace L \opid$. Since in this case we set $\tilde G_\kp=L\opid$, we have $ V(\hat G_\kp)\geq V(\tilde G_\kp)$. According to Lemma~\ref{lemma_grad:bound}, we have $\trace G_\kp\leq n\bar\kappa$. Since $\lambda_k\geq0$, we deduce that
    \begin{equation}\label{eq:lemma_grad:maindescent:proof2}
    \begin{split}
        V(G_\kp)-V(\tilde G_\kp)&= V(G_\kp)-V(\hat G_\kp) +V(\hat G_\kp)-V(\tilde G_\kp)\\&\geq V(G_\kp)-V(\hat G_\kp)  \\
       &\geq \trace(G_\kp-G_\kp-\lambda_\kp )\\
       &\geq -n\lambda_\kp\\
        &\geq - n\lambda_\kp\,.\\
    \end{split}
    \end{equation}
    \end{enumerate}
    Therefore, adding \eqref{ineq_grad:V3} with either \eqref{eq:lemma_grad:maindescent:proof1} or \eqref{eq:lemma_grad:maindescent:proof2} yields the desired inequality.
\end {proof}
Now, we are ready to prove our main theorem for Algorithm~\ref{Alg:Grad_PQN}. 
\begin {proof}[\textbf{Proof of Theorem~\ref{thm:main-grad}}]
For symplicity, we denote $\gamma_k^2 \coloneqq\frac{g^2_\kp}{a_k}$, where $a_k =S_1 g_k$ for general $\phi$ and $a_k ={\min\{S_1g_k,S_2g_k^{1/\theta}\}}$ for $\phi(t)=ct^{1-\theta}$.
    Summing both sides of \eqref{ineq:lemma_grad:maindescent_1} and \eqref{ineq:lemma_grad:maindescent_2} from $k=k_0$ to $N-1+k_0$, we obtain 
    \begin{equation}
        V(\tilde G_{k_0})-V(\tilde G_{N+k_0})\geq\sum_{k=k_0}^{N-1+k_0} \gamma_k^2 - n\sum_{k=1+k_0}^{N+k_0} \lambda_k\,.
    \end{equation}
    Since, for every $k$, our method keeps $\tilde G_k\succeq J_k$ and $G_\kp\succeq J_k$ (see Lemma~\ref{lemma_grad:bound}),  we have $V(\tilde G_k)>0$ and therefore
    \begin{equation}\label{ineq_grad:sumineq}
    \begin{split}
    V(\tilde G_{k_0})&\geq \sum_{k=k_0}^{N-1+k_0} \gamma_k^2 - n\sum_{k=1+k_0}^{N+k_0} \lambda_k\,.\\
    \end{split}
    \end{equation}
    By Lemma~\ref{lemma_grad:growthrk}, we obtain
    \begin{equation}\label{mainstep2_grad}
    V(\tilde G_{k_0})+C_{\mathrm{GR}}N^{2/3}\geq V(\tilde G_{k_0})+  nDC_0^{1/3} N^{2/3}\geq \sum_{k=k_0}^{N-1+k_0}\gamma_k^2\,,
    \end{equation}
    where $C_{\mathrm{GR}}\coloneqq nDC_0^{1/3}$. Dividing  \eqref{mainstep2_grad} by $N$, we obtain 
    \begin{equation}\label{mainstep2.1_grad}
    \frac{V(\tilde G_{k_0})}{N}+\frac{C_{\mathrm{GR}}}{N^{1/3}}\geq \frac{1}{N}\sum_{k=k_0}^{N-1+k_0}\gamma_k^2\,.
    \end{equation}
    We derive from \eqref{mainstep2.1_grad} with the concavity and monotonicity of $\log x$ that
    \begin{equation}\label{ineq:key_grad}
    \begin{split}
          \log\left(\frac{V(\tilde G_{k_0})}{N}+\frac{C_{\mathrm{GR}}}{N^{1/3}}\right) &\geq \log\left(\frac{1}{N}\sum_{k=k_0}^{N-1+k_0}\gamma_k^2\right)\geq\frac{1}{N}\sum_{k=k_0}^{N-1+k_0}\log(\gamma_k^2)= \log \left(\left(\prod_{k=k_0}^{N-1+k_0}\frac{g_\kp^2}{a_k}\right)^{1/N}\right)\\
          %=\log\left(\left(\frac{g_N}{g_0}\right)^{2/N}\right)\,.\\
    \end{split}
    \end{equation}
    We now discuss the two cases of Lemma~\ref{lemma_grad:maindescent} separately, depending on the desingularizing function $\phi$.  
    For general $\phi$, we have $a_k= S_1 g_k$. Thus, \eqref{ineq:key_grad} reads 
    \begin{equation}
          \log\left(\frac{V(\tilde G_{k_0})}{N}+\frac{C_{\mathrm{GR}}}{N^{1/3}}\right) 
          %\geq\log \left(\left(\prod_{k=k_0}^{N-1+k_0}\frac{g_\kp^2}{a_k}\right)^{1/N}\right)
          \geq \log\left(\left(\frac{ g_N(\prod_{k=1+k_0}^{N+k_0} g_k) }{S_1^Ng_0}\right)^{1/N}\right)\,.
    \end{equation}
    
    Therefore, we obtain 
    \begin{equation}
         \left(\min_{\set{k\in\N\vert k_0\leq i\leq N+k_0}} g_k\right)^{\frac{N+1}{N}}\leq \left(g_N^{1/N}\right) \left(\prod_{k=1}^{N} g_k\right)^{1/N}\leq \left(S_1\left(\frac{V(\tilde G_{k_0})}{N}+\frac{C_{\mathrm{GR}}}{N^{1/3}}\right)\right)g_{k_0}^{1/N}\,.
    \end{equation}
    Taking the power $N/(N+1)$ on both sides shows the desired claim in \eqref{convx_conv_Rate:KL}. 
    \begin{equation}
  \left(\min_{\set{k\in\N\vert k_0\leq k\leq N+k_0}} g_k\right)\leq \left(S_1\left(\frac{V(\tilde G_{k_0})}{N}+\frac{C_{\mathrm{GR}}}{N^{1/3}}\right)\right)^{N/(N+1)}g_{k_0}^{1/(N+1)}\,.
    \end{equation}
    For $\phi(t) = ct^{1/2}$, we have $a_k= \frac{c^2}{2}g_k^2$.
    \begin{equation}
        g_N\leq \left(\frac{c^2}{2}\left(\frac{V(\tilde G_{k_0})}{N}+\frac{C_{\mathrm{GR}}}{N^{1/3}}\right)\right)^{N/2}g_{k_0}\,.
    \end{equation}
    For $\phi(t)=ct^{1-\theta}$ with $\theta\in(0,\frac{1}{2})$, we have $a_k ={\min\{S_1g_k,S_2g_k^{1/\theta}\}}$. 
    Let $A = S_1g_k$ and $B=S_2g_k^{1/\theta}$. We notice that by Jensen's inequality.
    $$
    \min\set{A,B}\leq A^{\frac{1-2\theta}{1-\theta}} B^{ \frac{\theta}{1-\theta}}\,.
    $$
    Then, $a_k\leq S g_k^2$ where $S = S_1^{ \frac{1-2\theta}{1-\theta}} S_2^{ \frac{\theta}{1-\theta}}$.

    By \eqref{ineq:key_grad}, we have 
     \begin{equation}
          \log\left(\frac{V(\tilde G_{k_0})}{N}+\frac{C_{\mathrm{GR}}}{N^{1/3}}\right) 
          \geq\log \left(\left(\prod_{k=k_0}^{N-1+k_0}\frac{g_\kp^2}{a_k}\right)^{1/N}\right)
          \geq \log\left(\frac{1}{S}\left(\frac{ g_N }{ g_{k_0}}\right)^{2/N}\right)\,.
    \end{equation}
    Therefore, we obtain
     \begin{equation}
        g_N\leq\left(S\left(\frac{V(\tilde G_{k_0})}{N}+\frac{C_{\mathrm{GR}}}{N^{1/3}}\right)\right)^{N/2}g_{k_0}
    \end{equation}
    For $\phi(t)=ct^{1-\theta}$ with $\theta\in(\frac{1}{2},1)$, we have $a_k \leq{\min\{S_1g_k,S_2g_k^{1/\theta}\}}\leq S_2g_k^{1/\theta} $.
    Thus, 
    \begin{equation}
          \log\left(\frac{V(\tilde G_{k_0})}{N}+\frac{C_{\mathrm{GR}}}{N^{1/3}}\right) \geq\log \left(\left(\prod_{k=k_0}^{N-1+k_0}\frac{g_\kp^2}{a_k}\right)^{1/N}\right)
          =\log\left(\left(\frac{1}{S_2}\right)\left(\frac{ g^2_N(\prod_{k=1+k_0}^{N-1+k_0} g_k^{2-\frac{1}{\theta}}) }{g_{k_0}^{1/\theta}}\right)^{1/N}\right)\,.
    \end{equation}
    Thus, we obtain

    \begin{equation}
        \min_{\set{k\in\N\vert k_0\leq k\leq N+k_0}} g_k\leq \left( S_2\left(\frac{V(\tilde G_{k_0})}{N}+\frac{C_{\mathrm{GR}}}{N^{1/3}}\right)  g_{k_0}^{1/(\theta N)} \right)^{1/(\frac{2}{N}+(N-1)(2-\frac{1}{\theta}))}\,.\qedhere
    \end{equation}
    
\end {proof}
\subsubsection{Proof of Theorem~\ref{thm2:main-grad}}
When $g=0$, by Lemma~\ref{lemma_grad:descentVal}, we can derive the following bound on $\norm[]{\nabla f(x_k)}$ for any $k\in\N$.
\begin{lemma}\label{lemma:nabla_f_f-inf_f}
For any $k\in\N$, we have 
\begin{equation}
    \norm[]{\nabla f(x_k)}^2\leq 2n\bar \kappa \left(f(x_k)-f(x_\kp)\right)\,.
\end{equation}
\end{lemma}
\begin {proof}
    Thanks to Lemma~\ref{lemma_grad:bound}, we have $$\scal{\tilde G_ku_k}{u_k}= \scal{\tilde G_ku_k}{\tilde G_k^{-1}\tilde G_k u_k}\geq \frac{1}{n\bar\kappa}\scal{\tilde G_ku_k}{\tilde G_ku_k}= \frac{1}{n\bar\kappa}\norm[]{\nabla f(x_k)}^2\,.$$ Together with Lemma~\ref{lemma_grad:descentVal}, we obtain the desired result.
\end {proof}
\begin{lemma}\label{lemma_grad:bound_rk_nabla_f}
Assume $f$ satisfies \eqref{PL}. For any $N\in\N$, we have
    \begin{gather}
        f(x_N)-\textcolor{black}{\min f}\leq q^N(f(x_0)-\textcolor{black}{\min f})\\
        \sum_{k=0}^{N} \norm[]{\nabla f(x_k)}^{1/2}\leq C_f\,,\\
        \sum_{k=0}^{N} r_k \leq C_r
    \end{gather}
    where $q=\frac{cn\bar\kappa-1}{cn\bar\kappa} $, $C_f=4n\bar\kappa(cn\bar\kappa(f(x_0)-\textcolor{black}{\min f}))^{1/4}$ and $$C_r=\frac{1}{\sqrt{L_H}}C_f +  \frac{1}{L}\bar C_f\,.$$
\end{lemma}
\begin {proof}
    By Lemma~\ref{lemma:nabla_f_f-inf_f} and \eqref{PL}, we have that 
    \begin{equation}
    f(x_k)-\textcolor{black}{\inf f} \leq cn\bar\kappa \left(f(x_k) - f(x_\kp)\right)\leq cn\bar\kappa [f(x_k) - \textcolor{black}{\min f}- (f(x_\kp)-\textcolor{black}{\min f})]\,.
    \end{equation}
    Here, $cn\bar\kappa>1$. Otherwise, $f(x_\kp)\leq \min f$ and we get a contradiction.
    Then, we deduce that 
    \begin{equation}\label{ineq:f_inf}
        f(x_\kp)-\textcolor{black}{\min f}\leq q (f(x_k)-\textcolor{black}{\min f})\,,
    \end{equation}
    where $q=\frac{cn\bar\kappa-1}{cn\bar\kappa}$.
    Thus, $ f(x_N)-\textcolor{black}{\min f} \leq q^N (f(x_0)-\textcolor{black}{\min f})$.
    By Lemma~\ref{lemma:nabla_f_f-inf_f}, we derive from \eqref{ineq:f_inf} that
    \begin{equation}\label{ineq:nabla_f_qN}
        \norm[]{\nabla f(x_N)}^{1/2}\leq (2n\bar\kappa(f(x_0)-\textcolor{black}{\min f}))^{1/4} q^{N/4}\,.
    \end{equation}
    Summing \eqref{ineq:nabla_f_qN} up to $N$, we have
    \begin{gather}
        \sum_{0}^{N}\norm[]{\nabla f(x_N)}^{1/2}\leq \frac{1}{1-q^{1/4}}(2n\bar\kappa(f(x_0)-\textcolor{black}{\min f}))^{1/4}\,,\\
        \sum_{0}^{N}\norm[]{\nabla f(x_N)}\leq \frac{1}{1-q^{1/2}}(2n\bar\kappa(f(x_0)-\textcolor{black}{\min f}))^{1/2}\,.
    \end{gather}
    Since $\frac{1}{1-q^{1/4}}\leq \frac{4}{1-q}=4cn\bar\kappa$, we set $C_f = 4cn\bar\kappa(2n\bar\kappa(f(x_0)-\textcolor{black}{\min f}))^{1/4}$. Similarly, we set $\bar C_f= 2cn\bar\kappa(2n\bar\kappa(f(x_0)-\textcolor{black}{\min f}))^{1/2}$ .
    Thanks to the restarting strategy, we have $\tilde G_N=L$ or $\tilde G_N=G_N+\lambda_N\opid$ and $G_N\succeq 0$ for any $N\in\N$. Then, we deduce that $\tilde G_N\succeq \min\{L,\lambda_N\}\opid$. Thus, 
    \begin{equation}\label{ineq:minL_lbdN}
        \min\{L,\lambda_N\}\norm[]{u_N}\leq \norm[]{\tilde G_N u_N}\leq \norm[]{\nabla f(x_N)}\,.
    \end{equation}
    Since $\lambda_N>\sqrt{L_H\norm[]{\nabla f(x_N)}}$, \eqref{ineq:minL_lbdN} implies that 
    \begin{equation}\label{ineq:r_N_nabla_f}
    \begin{split}
         r_N&=\norm[]{u_N}\\
         &\leq \max\{\frac{1}{L}\norm[]{\nabla f(x_N)},\frac{1}{\lambda_N}\norm[]{\nabla f(x_N)}\} \\
         &\leq  \max\{\frac{1}{L}\norm[]{\nabla f(x_N)}, \frac{1}{\sqrt{L_H}}\sqrt{\norm[]{\nabla f(x_N)}}\}\\&\leq \frac{1}{L}\norm[]{\nabla f(x_N)}+\frac{1}{\sqrt{L_H}}\sqrt{\norm[]{\nabla f(x_N)}}\,.\\
    \end{split}
    \end{equation}
    Summing up \eqref{ineq:r_N_nabla_f} from $0$ to $N$, we have 
    \begin{equation}
        \sum_{0}^N r_k\leq\frac{1}{\sqrt{L_H}} \sum_{0}^{N}\norm[]{\nabla f(x_N)}^{1/2} + \frac{1}{L}\sum_{0}^{N}\norm[]{\nabla f(x_N)}\leq \frac{1}{\sqrt{L_H}}C_f +  \frac{1}{L}\bar C_f\,.
    \end{equation}
    We set 
    \begin{equation}
    \begin{split}
        C_r&\coloneqq
        % c^{1/4}(n\bar\kappa)^{5/4}\left(\frac{2}{L}+\frac{4}{\sqrt{L_H}}\right)\left(1+ (2n\bar\kappa(f(x_0)-\textcolor{black}{\inf f}))^{1/4}\right)\\
        \frac{1}{\sqrt{L_H}}4cn\bar\kappa(2n\bar\kappa(f(x_0)-\textcolor{black}{\min f}))^{1/4}+ \frac{1}{L}2cn\bar\kappa(2n\bar\kappa(f(x_0)-\textcolor{black}{\min f}))^{1/2}\\
        &=\frac{1}{\sqrt{L_H}}C_f +  \frac{1}{L}\bar C_f\,.\\
    \end{split}
    \end{equation}
    This shows the desired results.
\end {proof}
\begin{lemma}\label{lemma_grad2:maindescent}
For any $k\in\N$, we have a descent inequality as the following:
\begin{equation}
    V(\tilde G_k)-V(\tilde G_\kp)\geq \frac{2 g^2_\kp}{cg_k^2}-n\lambda_\kp\,,
\end{equation}
where $g_k\coloneqq \norm[]{F^\prime(x_k)}$.
\end{lemma}
\begin {proof}
    By the optimality condition in \eqref{optimality condition_grad:SR1}, we have $\tilde G_k u_k+v_\kp= -\nabla f(x_k)$ and by definition of $J_k$, we have $J_ku_k=\nabla f(x_\kp)-\nabla f(x_k)$. 
Using Lemma~\ref{lemma:potentialfunction}, we have the following estimation:
    \begin{equation}
    \begin{split}
         \textcolor{black}{V(\tilde G_k)-V(G_\kp)=\nu(J_k,\tilde G_k,u_k)} &\geq \frac{\norm[]{\tilde G_k-J_k)u_k}^2}{u_k^\top \tilde G_k u_k} = \frac{\norm[]{\nabla f(x_\kp)}^2}{u_k^\top \tilde G_k u_k}\\
         &\geq \frac{\norm[]{F^\prime (x_\kp)}^2}{2(f(x_k)-f(x_\kp))}\geq \frac{2\norm[]{F^\prime (x_\kp)}^2}{c\norm[]{F^\prime(x_k)}^2}=\frac{2 g_\kp^2}{cg_k^2}\,,
    \end{split}
    \end{equation}
    where the second inequality holds due to \text{Lemma~\ref{lemma:nabla_f_f-inf_f}} and the third inequality holds due to \eqref{PL}. 
    The rest of this proof remains the same as Theorem~\ref{lemma_grad:maindescent}
\end {proof}
\begin {proof}[\textbf{Proof of Theorem~\ref{thm2:main-grad}}]
    For symplicity, we denote $\gamma_k^2 \coloneqq\frac{2g^2_\kp}{c^2g^2_k}$. 
    Summing the result in Lemma~\ref{lemma_grad2:maindescent} from $k=0$ to $N-1$, we obtain 
    \begin{equation}
        V(\tilde G_{0})-V(\tilde G_N)\geq\sum_{k=0}^{N-1} \gamma_k^2 - n\sum_{k=1}^{N} \lambda_k\,.
    \end{equation}
    Since, for every $k$, our method keeps $\tilde G_k\succeq J_k$ and $G_\kp\succeq J_k$ (see Lemma~\ref{lemma_grad:bound}),  we have $V(\tilde G_k)>0$ and therefore
    \begin{equation}\label{ineq_grad2:sumineq}
    \begin{split}
    V(\tilde G_{0})&\geq \sum_{k=0}^{N-1} \gamma_k^2 - n\sum_{k=1}^{N} \lambda_k\,.\\
    \end{split}
    \end{equation}
    By Lemma~\ref{lemma_grad:bound_rk_nabla_f}, we obtain
    \begin{equation}\label{mainstep2_grad2}
    \left(V(G_{0})+n\sqrt{L_H}C_f+nL_HC_r  \right)\geq \sum_{k=k_0}^{N-1+k_0}\gamma_k^2\,.
    \end{equation}
    We set $C_{\mathrm{GD}}\coloneqq ( V(G_{0})+n\sqrt{L_H}C_f+nL_HC_r  )$. Dividing  \eqref{mainstep2_grad2} by $N$, we obtain 
    \begin{equation}\label{mainstep2.1_grad2}
    \frac{C_{\mathrm{GD}}}{N^{}}\geq \frac{1}{N}\sum_{k=0}^{N-1}\gamma_k^2\,.
    \end{equation}
    We derive from \eqref{mainstep2.1_grad2} with the concavity and monotonicity of $\log x$ that
    \begin{equation}\label{ineq:key_grad2}
    \begin{split}
          \log\Big(\frac{C_{\mathrm{GD}}}{N}\Big) &\geq \log\left(\frac{1}{N}\sum_{k=0}^{N-1}\gamma_k^2\right)\geq\frac{1}{N}\sum_{k=0}^{N-1}\log(\gamma_k^2)= \log \left(\left(\prod_{k=0}^{N-1}\frac{2g_\kp^2}{c^2g^2_k}\right)^{1/N}\right)\\
    \end{split}
    \end{equation}
   Therefore, we have
    \begin{equation}
        g_N\leq \left(\frac{c^2C_{\mathrm{GD}}}{2N^{}}\right)^{N/2}g_{0}\,.
    \end{equation}
    
\end {proof}
\section{Experiments}
In the following section, we consider three applications from regression problems and image processing to provide numerical evidence about the superior performance of our algorithms on convex and nonconvex problems with K\L\ inequality. While our algorithms are capable of solving non-smooth additive composite optimization problems, we focus solely on smooth problems, where our convergence analysis is also novel. This is mainly because efficiently solving sub-problems in the non-smooth case requires further investigations, possibly along the lines of \cite{becker2019quasi}, which we plan to address in future work.  In this section, we compare our algorithms with various algorithms including the gradient descent method (GD), Nesterov accelerated gradient descent (NAG), the 
Heavy ball method (HB), cubic regularized Newton method (Cubic Newton) \cite{nesterov2006cubic} and gradient regularized Newton method (Grad Newton) \cite{mishchenko2023regularized}.  Here, we use the solver developed in \cite{mishchenko2023regularized} to solve the subproblems that include cubic terms.
\subsection{Quadratic function}
For this experiment, we test our methods: \textup{Grad SR1 PQN} (Algorithm \ref{Alg:Grad_PQN}) on the convex quadratic problem with kernel:
\begin{equation}
    \min_{x\in\R^n} f(x)\,,\quad \text{with}\ f(x) = \frac{1}{2}\norm[]{Ax -b}^2\,,
\end{equation}
where $A\in\R^{m\times n}$ with $m<n$ \textcolor{black}{has non-trivial kernel and $b\in\R^m$.} This function $f$ is convex  {but not strongly convex since the Hessian $A^\top A\in\R^{n\times n}$ has 0 eigenvalue}.  The Lipschitz constant is $L=\norm[]{A^\top A}$ and $L_H=0$. By calculation, we can show that the function $f$ is gradient dominated, i.e., for any $x\in \R^n$, $$f(x)-\inf f\leq c\norm[]{ \nabla f(x)}^2\,,$$ for some  $c>0$ (see \cite{karimi2016linear}). If there exists $y$ such that $\nabla f(y)=A^\top(Ay-b)=0$, then $f(y)=\inf f$ due to the convexity of $f$.   {Though this function is no longer coercive, the convergence guarantee remains valid with $L_H=0$ and $\lambda_k \equiv 0$ for any $k\in\N$, as predicted by Theorem~\ref{thm:main-cubic} or~\ref{thm:main-grad}.} In our experiment, $A\in\R^{m\times n}$ and $b\in\R^m$ are generated randomly with $m=250$ and $n=300$. Since $\lambda_k\equiv 0$ for all $k\in\N$, there is no distinction among Cubic-, Grad-, and classical SR1 methods. If the quasi-Newton metric $G_k$ is invertible, we utilize the Sherman-Morrison formula for efficient implementation. As depicted in Figure~\ref{fig:log-sum-exp}, the SR1 method achieves a super-linear rate of convergence even on the quadratic problem with a non-trivial kernel. Here, the stepsize of NAG and GD is set to $1/L$.
%\textcolor{red}{\st{while the starting point is far from the optimal solution,}}

\begin{figure}[htp]
    \centering
    \includegraphics[width=0.48\linewidth]{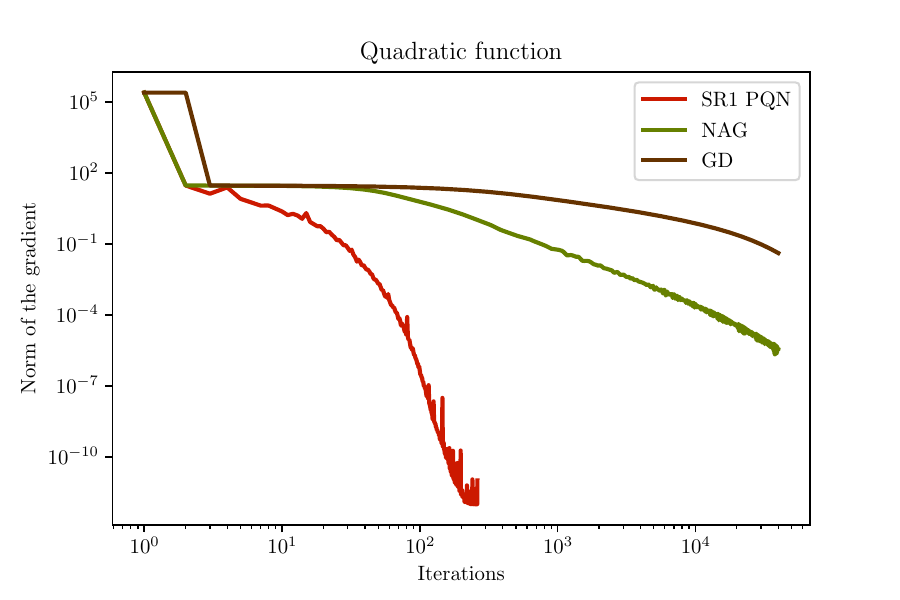}
    \includegraphics[width=0.48\linewidth]{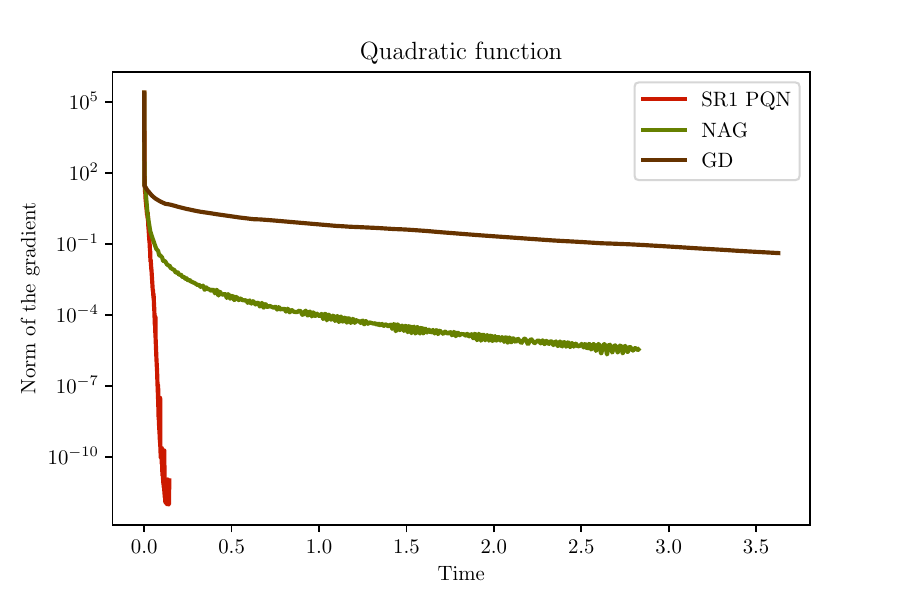}
    \caption{ \textup{SR1 method} significantly outperforms \textcolor{black}{first-order methods} in number of iterations and in terms of time. 
}
    \label{fig:log-sum-exp}
    %\label{fig:reg-SR1}
\end{figure}

\subsubsection{Logistic regression with a convex regularization}
The second experiment is a logistic regression problem with regularization on the benchmark dataset ``mushroom'' from UCI Machine Learning Repository \cite{mushroom_73}:
\begin{equation}
    \min_{x\in\R^n} f(x)\,,\quad \text{with}\ f(x)\coloneqq \frac{1}{m}\sum_{i=1}^m\log(1+\exp(-b_ia_i^\top x)) + \mu\sqrt{\norm[]{x}^2+\epsilon}\,,
\end{equation}
where $a_i\in\R^n$, $b_i\in\R$ for $i=1,2,\cdots,m$ denote the given data from the ``mushrooms'' dataset. We set $\epsilon=1$. The Lipschitz constants are $L=2\sum_{i=1}^m\norm[]{a_i}^2+\textcolor{black}{2\mu}$ and $L_H=\textcolor{black}{4\max_{i=1,\ldots,m} \norm[]{a_i}}$. However, in this experiment, we heuristically set $L_H=4$. Here, $m=8124$ and $n=117$. We set $\bar\kappa = 2L$ for \textup{Cubic- and Grad SR1 PQN}. The number of restarting steps for \textup{Grad SR1 PQN} is 1 and the number for \textup{Cubic SR1 PQN} is 14. We set the stepsize  as $1/L$ for NAG and GD. 
Note that $f$ is coercive and strongly convex on any compact sublevel set. Thus, K\L\ inequality holds with respect to $\phi(t)=ct^{1/2}$ for some $c>0$ and for any $x\in\set{x\vert f(x)\leq f(x_{0})}$. 
\begin{figure}[H]
    \centering
    \includegraphics[width=0.48\linewidth]{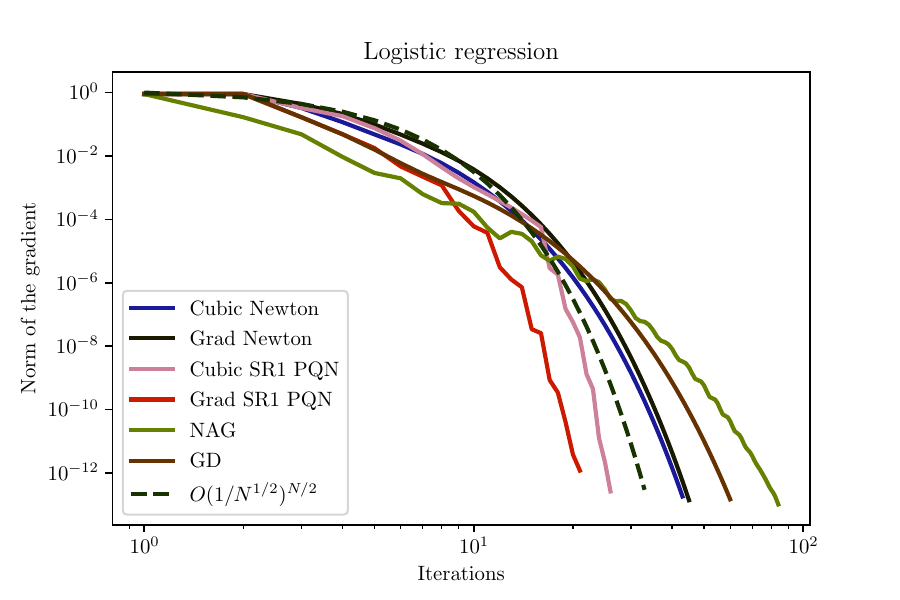}
    \includegraphics[width=0.48\linewidth]{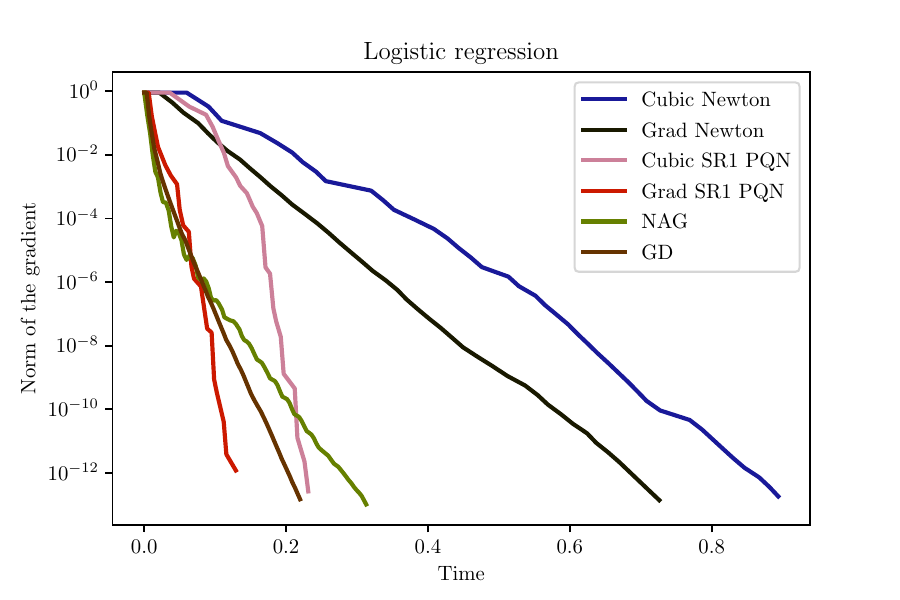}    \caption{\textup{Cubic- and Grad SR1 PQN}  significantly outperform first-order methods in number of iterations, exhibiting super-linear convergence. In terms of time, \textup{Grad SR1 PQN} remains one of the fastest methods due to the SR1 metric and the cheap computation of both the correction step and the restarting step.}
    \label{fig:reg-SR1-logistic}
\end{figure}
As depicted in Figure~\ref{fig:reg-SR1-logistic}, our two regularized quasi-Newton methods achieve a super-linear rate of convergence.  
% Cubic quasi-Newton requires only first-order information which provides a possibility to apply second order method in the case that the Hessian is not easy to evaluate. 
However, due to the cubic term, neither Cubic Newton nor Cubic quasi-Newton is competitive when it comes to measuring the actual computation time. Grad SR1 PQN outperforms all other first-order methods significantly in number of iterations and is among the fastest in terms of time.
\subsubsection{Image deblurring problem with a non-convex regularization}
The third experiment focuses on an image deblurring problem with non-convex regularization. Let $b\in\R^{MN}$ represent a blurry and noisy image, where each pixel $b_i$ falls within the range $[0,1]$ for any $i\in\set{1\ldots MN}$. Here, $b$ is constructed by stacking the M columns of length N. To reccover a clean image $x\in\R^{MN}$, we solve the following optimization problem:  
\begin{equation}
    \min_{x\in\R^n} f(x)\,,\quad \text{with}\ f(x)\coloneqq \frac{1}{2}\norm[]{Ax-b}^2 + \frac{\mu}{2}\log(\rho+\norm[]{Kx}^2)\,,
\end{equation}
where the blurring operator $A\in\R^{MN\times MN}$ is a linear mapping and $Kx$ computes the forward differences in horizontal and vertical direction with $\norm[]{K}\leq 2\sqrt{2}$. We set $\mu=0.001$ and $\rho=0.1$. The Lipschitz constants are estimated from above as $L=\norm[]{A}+5$ and $L_H=10$.  For this experiment, we take $M=N=64$. Additionally, we set $\bar\kappa = 2L$ for both \textup{Cubic- and Grad SR1 PQN}. We set the quasi-Newton metric as $\nabla^2 f(x_k)$ instead of $L\opid$ for both the initialization and the restarting steps. In this case, our convergence analysis remains valid. The number of restarting steps  for \textup{Cubic SR1 PQN} is $13$. For nonconvex optimization, we employ a line search with a backtracking strategy for both the gradient descent method (GD\_BT) and the Heavy ball method (HB\_BT), the latter being a special case of the iPiano algorithm from \cite{ochs2014ipiano}). Figure~\ref{fig:reg-SR1-debluring-ncvx} exhibits the comparison between our cubic SR1 method and  GD\_BT and HB\_BT .
\begin{figure}[H]
    \centering
    \includegraphics[width=0.48\linewidth]{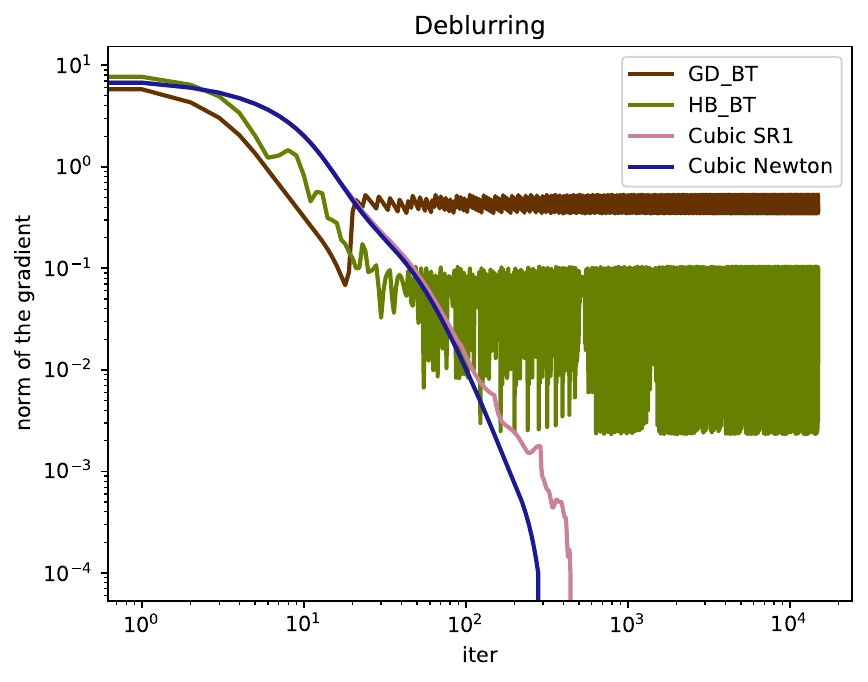}
    \includegraphics[width=0.48\linewidth]{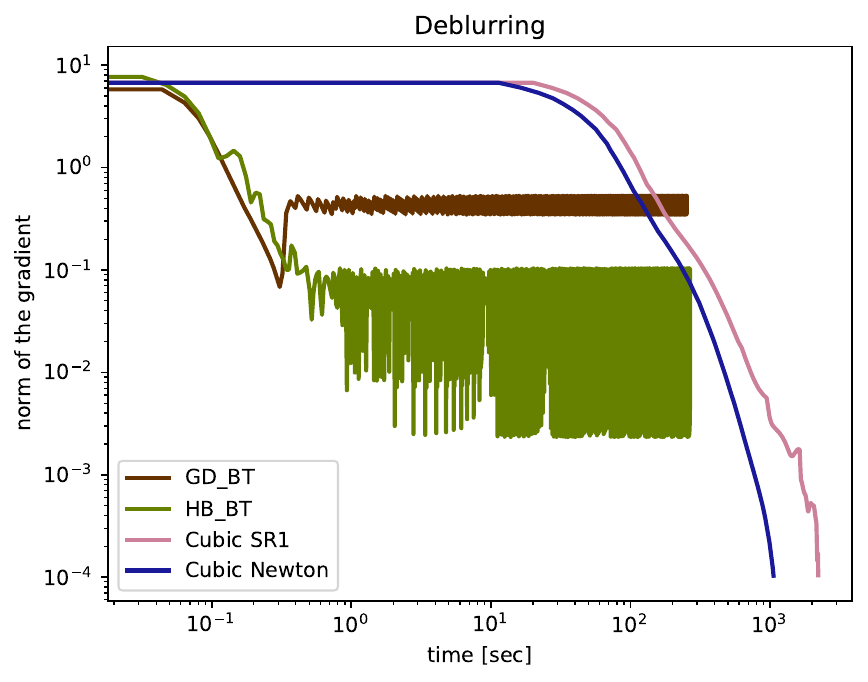}
    \caption{In the number of iterations and the norm of the gradients, Cubic Newton and \csr\ attain higher accuracy eventually. Moreover, both cubic Newton and \csr\ show more stable convergence, compared to first-order methods that utilize line search. }
    \label{fig:reg-SR1-debluring-ncvx}
\end{figure}
As illustrated in Figure~\ref{fig:reg-SR1-debluring-ncvx}, our \csr\ maintains superlinear convergence, despite slightly slower than the cubic Newton method. In the case when only the first-order information is accessible, \csr\ offers a more practical alternative. Both regularized second order methods are more stable than first-order method with Armijo line search when applied on this nonconvex problem. The oscillation in the latter methods may be attributed to the allowance of a large step size and a sharp local minimum of the nonconvex function. 
%As a result, despite the fact that the function value remains decreasing, the gradient norm oscillates.
\section{Conclusion}
 In this paper, we study two regularized proximal quasi-Newton SR1 methods which converge with explicit (non-asymptotic) super-linear rates \textcolor{black}{even for some classes of non-convex non-smooth optimization problems}. The key is the adaptation of convergence analysis from the work \cite{wang2024global} on strongly convex problems to non-convex problems that, instead, satisfies the K\L property, illustrating the power of the trace potential function $V(G)=\trace G$. Moreover, our algorithm and analysis is directly developed for non-smooth additive composite problems, although they are already novel in the smooth case. The analysis of convergence used in this paper, for the first time, reveals the possibility to study non-asymptotic super-linear convergence rates of classical quasi-Newton schemes without assuming strong convexity. This work shows that K\L\ property, which can be observed in many non-convex problems, is an important generalization of the strong convexity to study quasi-Newton methods and we can expect more results similar as in the strongly convex cases. 
\appendix
\section{Convergence of cubic quasi-Newton SR1 method with \L ojasiewicz inequality}\label{App:PL}
\begin{assum}\label{ass:main-1/2-L}
$f$ satisfies the global \L ojasiewicz inequality with constant $\mu$, i.e., there exists a constant $\mu$ such that for any $x\in\R^n$, we have
    \begin{equation}\label{ineq:lojasiewicz}
        f(x) - \min f\leq \frac{1}{2\mu}\norm[]{\nabla f(x)}^2\,.
    \end{equation}
\end{assum}
\begin{remark}
    This assumption is equivalent to assume $f$ has a quadratic growth rate for twice continuously differentiable $f$ \cite{rebjock2024fast}, i.e., $f(x)-\min f\geq c(\mathrm{dist}(x,\argmin f))^2$ for some constant $c>0$.  
\end{remark}
In order to solve the problem in \eqref{main:problem} with $g=0$ and $f$ satisfying \eqref{ineq:lojasiewicz}, we employ the same cubic regularized SR1 quasi-Newton method (Algorithm~\ref{Alg:SR1_da_PQN_PL}).  
\text{Algorithm~\ref{Alg:SR1_da_PQN_PL}} is the same with the first algorithm proposed in previous section where $g=0$.
\begin{algorithm}[H]
\caption{Cubic SR1 QN}\label{Alg:SR1_da_PQN_PL}
\begin{algorithmic}
\Require $x_0$, $G_0=L\opid$, $r_{-1}=0$.
\State \textbf{Update for $k = 0,\cdots,N$}:
\begin{enumerate}
\item\label{alg:SR1_da_PQN-PL-step1} \textbf{Update} \begin{equation*}\begin{split}
     x_\kp &\in \argmin_{x\in\R^n}\; \{ \scal{\nabla f(x_k)}{x-x_k}+\frac{1}{2}\norm[(G_k+ L_Hr_\km\opid)]{x-x_k}^2+{\frac{L_H}{3}\norm[]{x-x_k}^3}\}.\\
\end{split}
\end{equation*}
    \item \textbf{Set} $u_k=x_\kp-x_k$, $r_k=\norm[]{u_k}$, $\lambda_k= L_H(r_\km+r_k)$.
    \item\textbf{Compute} {$\tilde G_{k+1}=G_k+\lambda_k\opid$} and $$F^\prime(x_\kp)=\nabla f(x_\kp)-\nabla f(x_k)-\tilde G_{k+1}u_k\,.$$
    \item\textbf{Compute} $G_\kp=\text{SR1}( J_k,\tilde G_\kp,u_k)$ ($J_k\coloneqq\int_0^1\nabla^2f(x_k+tu_k)dt$).
    \item\textbf{If $\norm{F^\prime(x_\kp)}=0$}:
    Terminate.
\end{enumerate}
\State \textbf{End}
\end{algorithmic}
\end{algorithm}
\begin {theorem}\label{thm:main-cubic-PL}
Let Assumption~\ref{ass:main-1/2-L} hold and $g=0$. For any initialization $x_0\in\R^n$ and any $N\in\N$, \textup{Cubic SR1 QN} has a global convergence with the rate:
\begin{equation*}
    \norm[]{\nabla f (x_N)}\leq \frac{6}{\mu}\left(\frac{C}{N^{1/3}}\right)^{N/2}\norm[]{\nabla f (x_0)} \,,
\end{equation*}
where $C= \left( 2n L+n\textcolor{black}{R}+ 2nL_H \left(\frac{6}{L_H}\left(f(x_0)-\min f\right)\right)^{1/3} \right)$ and \textcolor{black}{$R=(6(f(x_0)-\min f)/L_H)^{1/3}$}.
\end {theorem}
\subsection{Convergence analysis of Algorithm~\ref{Alg:SR1_da_PQN_PL}}\label{conv:alg1-PL}
The optimality condition of the update step in \eqref{alg:SR1_da_PQN-PL-step1} is
\begin{equation}
   \nabla f(x_k) + (G_k+ L_Hr_\km\opid)(x_\kp-x_k) + L_H\norm[]{x_\kp-x_k}(x_\kp-x_k)=0\,,
\end{equation}
which, using our notations, simplifies to 
\begin{equation}
   \nabla f(x_k) + \underbrace{(G_k+\lambda_k\opid)}_{\tilde G_\kp}(x_\kp-x_k)=0\,.
\end{equation}
Surprisingly, using our notations, $\tilde G_\kp=G_k+\lambda_k\opid$.
\begin{lemma}\label{lemma:GgeqJ}
    For any $k\in\N$, we have
\begin{equation}
    {\tilde G_\kp}\succeq G_\kp\succeq J_k
\end{equation}
\end{lemma}
\begin {proof}
    We argue by induction. When $k=0$, we have  $G_0=L\opid\succeq J_0$ and $\tilde G_1\succeq J_0$. Then, $\tilde G_1\succeq G_1\succeq J_0$, thanks to  Lemma~\ref{lemma:Anton}. We assume the inequality holds for $k$, namely, $\tilde G_k\succeq G_k\succeq J_\km$. Thus, for $k+1$, we deduce that ${\tilde G_\kp}=G_k+L_H(r_k+r_\km)\opid\succeq J_\km + L_H(r_k+r_\km)\succeq J_k$. Thanks to Lemma~\ref{lemma:Anton}, we obtain ${\tilde G_\kp}\succeq G_\kp\succeq J_k$.
    Therefore, $\trace \tilde G_\kp\geq \trace G_\kp$.
\end {proof}

\begin{lemma}\label{lemma:decent-inequality-PL}
    For each $k\in\N$, we have
    \begin{gather*}
        f(x_\kp)-f(x_k)\leq -\frac{L_H}{6}r_k^3\,,\\
        \scal{({\tilde G_\kp}-J_k)u_k}{u_k}\leq \frac{6}{\mu}\norm[]{\nabla f(x_k)}^2\,.
    \end{gather*}
\end{lemma}
\begin {proof}
We start with showing the first result.
The update step~\ref{alg:SR1_da_PQN-PL-step1} implies that
    \begin{equation}\label{ineq:(1)}
         f(x_k)+\scal{\nabla f(x_k)}{u_k}+ \frac{1}{2}\scal{(G_k+L_Hr_\km\opid)u_k}{u_k}+\frac{L_H}{3}r_k^3\leq f(x_k)\,.
    \end{equation}
Then, thanks to Lemma~\ref{lemma:taylor}, we have
    \begin{equation}\label{ineq:(2)}
         f(x_\kp)-f(x_k)\leq \scal{\nabla f(x_k)}{u_k}  + {\frac{1}{2}}\scal{ \nabla^2 f(x_k)u_k}{u_k}+\frac{L_H}{6}\norm[]{u_k}^3\,.
 \end{equation}
 Combining two inequalities above, we can obtain that 
 \begin{equation}
 \begin{split}
    f(x_\kp)&\leq f(x_k) + \frac{1}{2}\scal{\nabla^2 f(x_k)u_k}{u_k}-\frac{1}{2}\scal{(G_k+L_Hr_\km\opid)u_k}{u_k}-\frac{L_H}{6} \norm[]{u_k}^3\\
    &\leq f(x_k) - \frac{L_H}{6}r_k^3\,,\\
 \end{split}
 \end{equation}
 where the last inequality holds thanks to the fact $G_k+L_H r_\km\opid\succeq J_\km +\frac{L_H}{2}r_\km\opid\succeq \nabla^2 f(x_k)$ (see Lemma~\ref{lemma:JleqJ}).
For the second result, we need another inequality from Lemma~\ref{lemma:taylor}:
\begin{equation}\label{ineq:(3)}
        f(x_\kp)-f(x_k)\leq \scal{\nabla f(x_k)}{u_k}  + {\frac{1}{2}}\scal{ J_k u_k}{u_k}+\frac{L_H}{6}\norm[]{u_k}^3\,.
    \end{equation}
Adding \eqref{ineq:(1)} to \eqref{ineq:(3)}, we obtain that
\begin{equation}
     \scal{({\tilde G_\kp}-J_k)u_k}{u_k}\leq 2(f(x_k)-f(x_\kp)) +\frac{2L_H}{3}r_k^3\,.
\end{equation}
Combining the first result above, we can derive that
\begin{equation}
    \begin{split}
        \scal{({\tilde G_\kp}-J_k)u_k}{u_k}&\leq 2(f(x_k)-f(x_\kp)) +\frac{2L_H}{3}r_k^3\\
        &\leq 6(f(x_k)-f(x_\kp))\leq{6(f(x_k)-\inf f)}\\
        &\leq  \frac{6}{\mu}\norm[]{\nabla f(x_k)}^2\,,\\
    \end{split}
\end{equation}
where the last inequality utilizes \L ojasiewicz ineqality.
\end {proof}

\begin{lemma}\label{lemma:summation-rk-PL}
Given a number of iterations $N\in\N$, we have: 
    \begin{gather}
    \label{ineq:lemma:summation-rk-PL}\sum_{k=0}^{N} r_k^3\leq \frac{6(F(x_0) -\inf F)}{L_H}<+\infty\,, \\
    \sum_{k=0}^{N-1} r_k\leq C_0 N^\frac{2}{3}\,, \\
    \sum_{k=1}^{N} r_k\leq C_0 N^\frac{2}{3}\,,
\end{gather}
where $C_0\coloneqq (\frac{ 6(F(x_0) -\inf F)}{L_H})^{1/3}$.
\end{lemma}
\begin {proof}
    The first inequality is derived from Lemma~\ref{lemma:decent-inequality-PL}. Then, we apply the Cauchy--Schwarz inequality and obtain
    \begin{equation}
        \sum_{i=0}^{N-1} r_k \leq \left(\sum_{k=0}^{N-1} r_k^3\right)^{1/3}\left(\sum_{i=0}^{N-1} 1^{\frac{3}{2}}\right)^{2/3}\leq\left(\sum_{k=0}^{\infty} r_k^3\right)^{1/3}N^{2/3}\leq C_0N^{2/3} \,.
    \end{equation}
    Similarly, we have 
    \begin{equation}
        \sum_{i=1}^{N} r_k \leq \left(\sum_{k=1}^{N} r_k^3\right)^{1/3}\left(\sum_{i=1}^{N} 1^{\frac{3}{2}}\right)^{2/3}\leq \left(\sum_{k=0}^{\infty} r_k^3\right)^{1/3}N^{2/3}\leq C_0N^{2/3} \,. 
    \end{equation}
    % \begin{equation}
    %     \sum_{i=0}^{N-1} r^2_k \leq (\sum_{k=0}^{N-1} r_k^3)^{2/3}(\sum_{i=0}^{N-1} 1^{3})^{1/3}\leq(\sum_{k=0}^{\infty} r_k^3)^{1/3}N^{1/3}\leq C_0 N^{1/3} \,.
    % \end{equation}
\end {proof}
\begin{lemma}
    \begin{equation}
    V(\tilde G_\kp)-V(\tilde G_{k+2})\geq \frac{\mu g^2_\kp}{6g_k^2}-n\lambda_\kp\,,
\end{equation}
where $g_k\coloneqq \norm[]{\nabla f(x_k)}$.
\end{lemma}
\begin {proof}
Wlog, we assume $(\tilde G_\kp-J_k)u\neq0$.
\begin{equation*}
     V(\tilde G_\kp)-V(G_\kp)=\nu(J_k, \tilde G_\kp,u_k) =
        \frac{u_k^\top( \tilde G_\kp-J_k)(\tilde G_\kp-J_k)u_k}{u_k^\top (\tilde G_\kp-J_k) u_k}\,,
    \end{equation*}

Using the definition of $\tilde G_\kp$, we have $(\tilde G_\kp-J_k) u_k = \nabla f(x_k)-\nabla f(x_\kp) -\nabla f(x_k)=-\nabla f(x_\kp)$. Thanks to the second inequality from Lemma~\ref{lemma:decent-inequality-PL}, we obtain
\begin{equation*}
    \nu(J_k, \tilde G_\kp,u_k) \geq \frac{ \mu\norm[]{\nabla f(x_\kp)}^2}{6\norm[]{\nabla f(x_k)}^2}\,.
\end{equation*}
On the other hand, 
\begin{equation*}
    V(G_\kp) - V(\tilde G_{k+2}) = -n\lambda_\kp\,.
\end{equation*}
\end {proof}
\begin {theorem}
For any initialization $x_0\in\R^n$ and any $N\in\N$, \textup{Cubic SR1 QN} has a global convergence with the rate:
\begin{equation*}
    \norm[]{\nabla f (x_N)}\leq \frac{6}{\mu}\left(\frac{C}{N^{1/3}}\right)^{N/2}\norm[]{\nabla f (x_0)} \,,
\end{equation*}
where $C= \left( 2n L+2nL_HR+ 2nL_H \left(\frac{6}{L_H}\left(f(x_0)-\min f\right)\right)^{1/3} \right)$. Here, $R=\left(\frac{6(F(x_0)-\inf F)}{L_H}\right)^{1/3}$.
\end {theorem}
\begin {proof}[Proof of Theorem~\ref{thm:main-cubic-PL}]
    Summing up the following inequalities from $k=0$ to $k=N-1$
     \begin{equation*}
    V(\tilde G_\kp)-V(\tilde G_{k+2})\geq \frac{\mu g^2_\kp}{6g_k^2}-n\lambda_\kp\,,
\end{equation*}
we have 
\begin{equation*}
    V(\tilde G_1)- V(\tilde G_{N}) + n\sum_{k=0}^{N-1} \lambda_k\geq \sum_{k=0}^{N-1} \frac{\mu g^2_\kp}{6g_k^2}\,.
\end{equation*}
Thanks to Lemma~\ref{lemma:GgeqJ},
$V(\tilde G_{N})\geq V(J_{N-1})\geq - nL$, $V(\tilde G_1)\leq nL+2nL_H\textcolor{black}{R}$. According to Lemma~\ref{lemma:summation-rk-PL},  $\sum_{k=1}^N\lambda_k\leq 2L_HC_0(N^{2/3})$ and \eqref{lemma:summation-rk-PL} implies that for any $k\in\N$, $r_k\leq R$ where $R=\left(\frac{6(F(x_0)-\inf F)}{L_H}\right)^{1/3} $.
    Dividing both sides by $N$ and using the concavity of $\log$, we have
\begin{equation}
\begin{split}
    \log\left(\frac{1}{N}\left(2nL+2nL_H\textcolor{black}{R} + 2L_HC_0(N^{2/3})\right)\right)\geq \log\left(\frac{1}{N}\left( \sum_{k=0}^N \frac{\mu g^2_\kp}{6g_k^2}\right)\right)&\geq \left(\frac{1}{N} \sum_{k=0}^N \log\left(\frac{\mu g^2_\kp}{6g_k^2}\right)\right)\\&\geq \log \left(\frac{\mu}{6}\left(\frac{g_N}{g_0}\right)^{2/N}\right)\,.\\
\end{split}
\end{equation}
\end {proof}
\section{A summable sequence}
\begin{lemma}[\cite{combettes2001quasi}]
\label{app:lemma:combettes}
Let  $\rho \in (0, 1)$, $(a_k)_{k\in\N}$ and $(b_k)_{k\in\N}$ be two nonnegative sequences such that $b_k$ is summable and
$a_{k+1} \leq (1-\rho)a_k + b_k\,.$
Then $a_k$ is summable. More precisely, $\sum_{k=0}^{+\infty} a_k \leq \rho^{-1}(a_0 + \sum_{k=0}^{+\infty} b_k)$.
\end{lemma}
\begin{proof}
    See \cite[Lemma 3.1]{combettes2001quasi}\,.
\end{proof}
\bibliography{main}
\bibliographystyle{siam}
\end{document}